\documentclass[a4paper,12pt]{amsart}

\usepackage[mathcal]{euscript}
\usepackage{mathrsfs}
\usepackage[shortlabels]{enumitem}
\usepackage[matrix,arrow]{xy}
\usepackage{comment}
\usepackage{amsmath}
\usepackage{amssymb,enumitem,verbatim,stmaryrd,xcolor,microtype,graphicx,needspace}
\usepackage[T1]{fontenc}
\usepackage[utf8]{inputenc}
\usepackage[english]{babel} 
\usepackage[top=3.3cm,bottom=3.3cm,left=3.25cm,right=3.25cm]{geometry}
\usepackage[bookmarksdepth=2,linktoc=page,colorlinks,linkcolor={red!80!black},citecolor={red!80!black},urlcolor={blue!80!black}]{hyperref}

\usepackage{tikz}\usetikzlibrary{matrix,arrows,decorations.markings}
\usepackage{tikz-cd}
\usetikzlibrary{automata,positioning,decorations.pathreplacing}
\newcommand\cprime\textquotesingle 
\usepackage{xcolor}

\usepackage[figuresleft]{rotating}
\usepackage{changepage}

\usepackage{resizegather}
\usepackage{multicol}
\tikzset{->-/.style={decoration={markings,mark=at position #1 with {\color{black}\arrow{>}}},postaction={decorate,very thick}}}
\tikzstyle{vertex}=[circle, draw, inner sep=0pt, minimum size=6pt]

\usepackage[mathcal]{euscript}                            
\usepackage{mathptmx}                                       
\usepackage[colorinlistoftodos,bordercolor=orange,backgroundcolor=orange!20,linecolor=orange,textsize=footnotesize]{todonotes} \makeatletter \providecommand \@dotsep{5} \def\listtodoname{List of Todos} \def\listoftodos{\@starttoc{tdo}\listtodoname} \makeatother

\allowdisplaybreaks 
\usepackage{etoolbox}
\makeatletter
\patchcmd{\@startsection}{\@afterindenttrue}{\@afterindentfalse}{}{}             
\patchcmd{\part}{\bfseries}{\bfseries\LARGE}{}{}
\patchcmd{\section}{\scshape}{\bfseries}{}{}\renewcommand{\@secnumfont}{\bfseries} 
\patchcmd{\@settitle}{\uppercasenonmath\@title}{\large}{}{}
\patchcmd{\@setauthors}{\MakeUppercase}{}{}{}
\addto{\captionsenglish}{} 
\addto{\captionsenglish}{} 
\addto{\captionsenglish}{} 
\makeatother

\usepackage{fancyhdr}

\theoremstyle{plain}
\newtheorem{thm}{Theorem}[section]
\newtheorem{cor}[thm]{Corollary}
\newtheorem{lemma}[thm]{Lemma}
\newtheorem{prop}[thm]{Proposition}
\newtheorem{thmA}{Theorem}

\newtheorem*{thm*}{Theorem}
\newtheorem*{lem*}{Lemma}

\theoremstyle{definition}
\newtheorem{df}[thm]{Definition}

\newtheorem{rem}[thm]{Remark}

\newtheorem{ex}[thm]{Example}
\newtheorem*{df*}{Definition}
\newtheorem*{ex*}{Example}
\newtheorem*{rem*}{Remark}

\DeclareRobustCommand{\gobblefour}[5]{}   

\DeclareFontFamily{OT1}{pzc}{}            
\DeclareFontShape{OT1}{pzc}{m}{it}{<-> s * [1.10] pzcmi7t}{}
\DeclareMathAlphabet{\mathpzc}{OT1}{pzc}{m}{it}
\DeclareSymbolFont{sfoperators}{OT1}{bch}{m}{n} \DeclareSymbolFontAlphabet{\mathsf}{sfoperators} \makeatletter\def\operator@font{\mathgroup\symsfoperators}\makeatother 
\DeclareSymbolFont{cmletters}{OML}{cmm}{m}{it}
\DeclareSymbolFont{cmsymbols}{OMS}{cmsy}{m}{n}
\DeclareSymbolFont{cmlargesymbols}{OMX}{cmex}{m}{n}
\DeclareMathSymbol{\myjmath}{\mathord}{cmletters}{"7C}     \let\jmath\myjmath 
\DeclareMathSymbol{\myamalg}{\mathbin}{cmsymbols}{"71}     
\DeclareMathSymbol{\mycoprod}{\mathop}{cmlargesymbols}{"60}
\DeclareMathSymbol{\myalpha}{\mathord}{cmletters}{"0B}     \let\alpha\myalpha 
\DeclareMathSymbol{\mybeta}{\mathord}{cmletters}{"0C}      \let\beta\mybeta
\DeclareMathSymbol{\mygamma}{\mathord}{cmletters}{"0D}     \let\gamma\mygamma
\DeclareMathSymbol{\mydelta}{\mathord}{cmletters}{"0E}     \let\delta\mydelta
\DeclareMathSymbol{\myepsilon}{\mathord}{cmletters}{"0F}   \let\epsilon\myepsilon
\DeclareMathSymbol{\myzeta}{\mathord}{cmletters}{"10}      \let\zeta\myzeta
\DeclareMathSymbol{\myeta}{\mathord}{cmletters}{"11}       \let\eta\myeta
\DeclareMathSymbol{\mytheta}{\mathord}{cmletters}{"12}     \let\theta\mytheta
\DeclareMathSymbol{\myiota}{\mathord}{cmletters}{"13}      \let\iota\myiota
\DeclareMathSymbol{\mykappa}{\mathord}{cmletters}{"14}     \let\kappa\mykappa
\DeclareMathSymbol{\mylambda}{\mathord}{cmletters}{"15}    \let\lambda\mylambda
\DeclareMathSymbol{\mymu}{\mathord}{cmletters}{"16}        \let\mu\mymu
\DeclareMathSymbol{\mynu}{\mathord}{cmletters}{"17}        \let\nu\mynu
\DeclareMathSymbol{\myxi}{\mathord}{cmletters}{"18}        \let\xi\myxi
\DeclareMathSymbol{\mypi}{\mathord}{cmletters}{"19}        \let\pi\mypi
\DeclareMathSymbol{\myrho}{\mathord}{cmletters}{"1A}       \let\rho\myrho
\DeclareMathSymbol{\mysigma}{\mathord}{cmletters}{"1B}     \let\sigma\mysigma
\DeclareMathSymbol{\mytau}{\mathord}{cmletters}{"1C}       \let\tau\mytau
\DeclareMathSymbol{\myupsilon}{\mathord}{cmletters}{"1D}   \let\upsilon\myupsilon
\DeclareMathSymbol{\myphi}{\mathord}{cmletters}{"1E}       \let\phi\myphi
\DeclareMathSymbol{\mychi}{\mathord}{cmletters}{"1F}       \let\chi\mychi
\DeclareMathSymbol{\mypsi}{\mathord}{cmletters}{"20}       \let\psi\mypsi
\DeclareMathSymbol{\myomega}{\mathord}{cmletters}{"21}     \let\omega\myomega
\DeclareMathSymbol{\myvarepsilon}{\mathord}{cmletters}{"22}\let\varepsilon\myvarepsilon
\DeclareMathSymbol{\myvartheta}{\mathord}{cmletters}{"23}  \let\vartheta\myvartheta
\DeclareMathSymbol{\myvarpi}{\mathord}{cmletters}{"24}     \let\varpi\myvarpi
\DeclareMathSymbol{\myvarrho}{\mathord}{cmletters}{"25}    \let\varrho\myvarrho
\DeclareMathSymbol{\myvarsigma}{\mathord}{cmletters}{"26}  \let\varsigma\myvarsigma
\DeclareMathSymbol{\myvarphi}{\mathord}{cmletters}{"27}    \let\varphi\myvarphi

\DeclareMathOperator{\Hom}{Hom}
\DeclareMathOperator{\Aut}{Aut}

\DeclareMathOperator{\GL}{GL}
\DeclareMathOperator{\PGL}{PGL}

\DeclareMathOperator{\supp}{supp}
\DeclareMathOperator{\rk}{rk}
\DeclareMathOperator{\Proj}{Proj}

\DeclareMathOperator{\Gr}{Gr}
\DeclareMathOperator{\Coh}{{Coh}}
\DeclareMathOperator{\Bun}{{Bun}}
\DeclareMathOperator{\PBun}{{\mathbb{P}Bun}}
\DeclareMathOperator{\Pic}{{Pic}}
\DeclareMathOperator{\Ext}{{Ext}}
\DeclareMathOperator{\Iso}{{Iso}}

\newcommand\A{{\mathbb A}}

\newcommand\C{{\mathbb C}}

\newcommand\FF{{\mathbb F}}
\newcommand\F{{\mathcal F}}
\newcommand\G{{\mathcal G}}

\newcommand\N{{\mathbb N}}

\renewcommand\P{{\mathbb P}}

\newcommand\Z{{\mathbb Z}}

\newcommand\cA{{\mathcal A}}

\newcommand\cE{{\mathcal E}}
\newcommand\cF{{\mathcal F}}
\newcommand\cG{{\mathcal G}}

\newcommand\cH{{\mathcal H}}

\newcommand\cK{{\mathcal K}}
\newcommand\cL{{\mathcal L}}

\newcommand\cO{{\mathcal O}}

\newcommand{\E}{\mathcal E}
\newcommand{\Line}{\mathcal L}
\newcommand{\Fq}{\mathbb{F}_q}
\makeatletter
\DeclareRobustCommand\bigop[1]{%
  \mathop{\vphantom{\sum}\mathpalette\bigop@{#1}}\slimits@
}
\newcommand{\bigop@}[2]{%
  \vcenter{%
    \sbox\z@{$#1\sum$}%
    \hbox{\resizebox{\ifx#1.9\fi\dimexpr\ht\z@+\dp\z@}{!}{$\m@th#2$}}%
  }%
}
\makeatother

\newcommand{\hprod}{\DOTSB\bigop{*}}

\renewcommand\geq{\geqslant}
\renewcommand\leq{\leqslant}

\renewcommand\emptyset\varnothing

\title{Hall algebras and Hecke modifications of vector bundles}

\author{Roberto Alvarenga}
\address{\rm Roberto Alvarenga, São Paulo State University (UNESP), São José do Rio Preto, Brazil}
\email{{roberto.alvarenga@unesp.br}}

\author{Leonardo Moço}
\address{\rm  Instituto de Ci\^encias Matem\'aticas e de Computa\c{c}\~ao - USP, S\~ao Carlos, Brazil}
\email{leonardo.moco@icmc.usp.br}

\allowdisplaybreaks

\begin{document}

\begin{abstract}
  In this article, we investigate Hecke modifications of vector bundles on a smooth projective curve $X$ defined over an arbitrary field. We obtain structural results that allow us to reduce the classification problem of Hecke modifications to the case of vector bundles of lower rank. Moreover, when the base field is a finite field and $X$ is the projective line, we apply the Hall algebra of coherent sheaves to provide a full classification of the Hecke modifications, including their multiplicities. These results are applied to study the space of unramified automorphic forms for $\PGL_n$ over the projective line, leading to a proof that the space of unramified toroidal automorphic forms is trivial.
\end{abstract}

\maketitle

\tableofcontents
\section{Introduction}

Let $X$ be a smooth projective curve defined over an arbitrary field $k$. Let $D$ be an effective divisor on $X$. Given $\E, \E'$ two vector bundles (locally free sheaves) of the same rank over $X$, roughly speaking, $\E'$ is a Hecke modification of $\E$  
at $D$ if $\E'$ is contained in $\E$ (as a locally free sheaf) with $\E/\E'$ isomorphic to the structural sheaf at $D$. 

Hecke modifications of vector bundles have been investigated, at least implicitly, since Weil \cite{weil-38}. It has been played a key role on both algebraic geometry and number theory and is also known as ``elementary transformations'' or ``Hecke transform''. The name ``Hecke'' comes from its connection with number theory, as it is related to the action of Hecke operators on the space of automorphic forms, see \cite{harder-67}.

In algebraic geometry, Hecke modifications have been used as an important tool for decades, see e.g.\ \cite{narasimhan-ramanan-78}. In \cite{mehta-seshadri-80}, a connection between Hecke modifications and the parabolic structure of a given vector bundle is established. This connection continues to be explored, as is evident in recent works such as \cite{alfaya-gomez-21} and \cite{he-walpuski-19}. 
In \cite{Araujo2021automorphisms}, the authors show that the admissible Hecke modifications generate the automorphism group of the moduli space of semistable parabolic bundles of rank $2$ over $\P^1 $ with trivial determinant. In \cite{HeuLoray-19}, the authors apply the Hecke modifications to study parabolic bundles with logarithmic connections. 
In \cite{boozer-21}, the author explicitly computes the Hecke modifications of rank $2$ vector bundles when $X$ is either the projective line or an elliptic curve. In this setting, he constructs a canonical open embedding from the moduli space of Hecke modifications of parabolic bundles into the moduli space of stable parabolic bundles with trivial determinant and a fixed number of marked points.

In number theory, i.e.,\ when $k$ is a finite field, the explicit descriptions of Hecke modifications allow us to explicitly calculate the action of Hecke operators on automorphic forms, and a connection with geometric Langlands program is established. This has been used to investigate the space of (unramified) automorphic forms and its subspaces spanned by eigenforms, cusp forms and toroidal forms, see e.g. \cite{oliver-elliptic} and \cite{rov-elliptic1}.   We refer to \cite{alvarenga-kaur-moço-25} for a complete discussion about Hecke modifications and its appearances across several branches of mathematics. 

\subsection*{Intention and scope of this article} In this work, we consider the case where the previous $D$ is supported in a single closed point $x \in X$. Let $|x|$ stand for the degree of $x$ and $\cK_x$ stand for its skyscraper sheaf. 

Let $r \in \Z_{>0}$. If $\E,\E' \in \Bun_n (X)$ are such that $\E'\subseteq \E$ and $\E/\E' \cong \cK_x^{\oplus r}$, we say that $\E'$ is a Hecke modification of $\E$ at $x$ with weight $r$. We denote such (isomorphism class of) Hecke modification by $[\E' \xrightarrow[r]{x} \E]$. 

Let $\F$ be a coherent sheaf  on $X$, we denote $\mu(\F) := \deg(\F)/\rk(\F)$ by the slope of $\F$. Moreover, all stability conditions in this article refer to $\mu$-stability. While our main goal is to explicitly describe Hecke modifications in the case where $X=\P^1$, we also provide structural results for a general curve. As, for example, Theorem \ref{breaksemistable}, which we state below:

\begin{thmA} \label{thm-A}
 Let $x \in X$ be a closed point, and let $\E, \E'\in \Bun_n(X)$. Write 
 $ \E:= \bigoplus_{i=1}^{m_1}\F_i$ and  $\E':= \bigoplus_{i=1}^{m_2}\F'_i$
 where $\cF_i$ and $\cF'_i$ are semistable bundles with $\mu(\cF_i) \leq \mu(\cF_{i+1})$ and $\mu(\cF_j ') \leq \mu(\cF_{j+1}')$ for $i= 1, \ldots, m_1$ and $j=1, \ldots, m_2$. Suppose that there exist $a,b \in \Z_{>0}$ such that
\[\E_1=\bigoplus_{i=1}^{a-1}\cF_i, \ \ \E_2=\bigoplus_{i=a}^{m_1}\cF_i, \ \ \E_1'=\bigoplus_{i=1}^{b-1}\cF_i'\ \ \text{and} \ \ \E_2'=\bigoplus_{i=b}^{m_2}\cF_i'\]
with $\rk(\E_1')=\rk(\E_1)$ and
$\mu(\cF_j')>\mu(\cF_i)$ for $j = b, \dots, m_2$ and $i = 1, \dots, a-1$.
Then $[\E' \xrightarrow[r]{x} \E]$ is a Hecke modification if and only if $[\E_1' \xrightarrow[r_1]{x} \E_1]$ and $[\E_2' \xrightarrow[r_2]{x} \E_2]$ are Hecke modifications, where $r_1 :=(\deg(\E_1')-\deg(\E_1))/|x|$ and $r_2 :=r-r_1$.
\end{thmA}

Theorem \ref{thm-A} enables the construction of Hecke modifications from vector bundles of lower rank. Moreover, as a consequence, it yields a complete classification of Hecke modifications in the case where $X$ is the projective line over an algebraically closed field, see either Theorem \ref{thm-heckemodforP1anyfield}  or the Theorem \ref{teo2} below. 

If $k = \FF_q$ is a finite field, then $\Ext_{\Coh (X)}(\E',\cK_x^{\oplus r})$ is a finite set. In this setting, we are particularly interested in computing the multiplicity $m_{x,r}(\E', \E)$ of the Hecke modification $[\E' \xrightarrow[r]{x} \E]$, i.e., the number of distinct Hecke modifications $[\E'' \xrightarrow[r]{x} \E]$ with $\E'' \cong \E'$. As noted in \cite[Lemma 2.1]{alvarenga20}, in this case, one can obtain the Hecke modifications and their multiplicities from certain products in the Hall algebra of $\Coh(X)$. When $X$ is the projective line, we apply the structure results from \cite{baumann-kassel-01} to obtain a full classification of the Hecke modifications, including their multiplicities. In the following, we state our main results in this direction.
 
\begin{thmA}\label{teo2} 
        Let $x \in \P^1$ be a closed point of degree one, and let $\E \in \Bun_n\P^1$ be such that $\E = \bigoplus_{i=1}^n \cO(d_i)$. Then $\E'\in \Bun_n(\P^1)$ is a Hecke modification of $\E$ at $x$ of weight $r$ if and only if there exists a function  
        \[ \delta: \{1, \ldots, n\} \to \{0,1\}, \]  
        supported on $r$ indices, such that $ \E' \cong \bigoplus_{i=1}^n \cO(d_i-\delta(i))$. Furthermore, when $k=\FF_q$, we write  $\E \cong \bigoplus_{i=1}^m \bigoplus_{j=1}^{\ell_i} \cO(b_i)$, then
    \[m_{x,r}(\E', \E)=q^{\alpha}\prod_{j=1}^m \#\Gr(\theta_j,\ell_j),\]
    where $\theta_j=\#\{ i \in \{1, 2, \ldots, n\}\,|\, d_i=b_j \text{ and } \delta(i)=1\}$ and $ \alpha=\sum_{j=1}^m(\ell_j-\theta_j)(r-\sum_{i=1}^j \theta_i)$.
\end{thmA}
This is Lemma \ref{existtrivial}
 together with Theorem \ref{deg1anyr}.
In particular, Theorem \ref{teo2} leads with the classification of all Hecke modifications of vector bundles in the projective line over any algebraically closed field.

For closed points of higher degree in $\P^1$, when $k = \Fq$, the following theorem, which is Theorem \ref{thm-C}, gives necessary and sufficient conditions for the existence of Hecke modifications. 

\begin{thmA}
Let  $x \in \P^1$ be a closed point of degree $d$, and let $\E', \E \in \Bun_n\P^1$ be such that $\E := \bigoplus_{i=1}^n \cO(d_i)$ and $\E' := \bigoplus_{i=1}^n \cO(d_i')$. Let us suppose that $d_i'=d_i-\epsilon_i$, where $ \epsilon_i \in \{0, 1, \ldots, d\}$ for all $i \in \{1, \ldots, n\}$ and $\sum_{i=1}^n \epsilon_i= d$. Let $A: = \{i \in \{1, \ldots, n\}\,\big|\, \epsilon_i \neq 0\}$, $s : =\min(A)$ and $B:=\max(A)$. Then $\E'$ is a Hecke modification of $\E$ at $x$ with weight $1$ if and only if 
    \[d_{j+1}-\epsilon_{j+1} \leq d_j \ \ \text{for all $j \in \{s, s+1, \ldots , B-1\}$}.\]
\end{thmA}

As previously explained, when \( k \) is a finite field, Hecke modifications describe the action of Hecke operators on the space of (unramified) automorphic forms. In this setting, we conclude the article by applying the calculations from Section \ref{sec-explicit calculations} to investigate the space of unramified automorphic forms for \( \PGL_n \) over \( \mathbb{P}^1 \). In Theorem \ref{thm-toroidal}, we prove that the space of unramified toroidal automorphic forms is trivial for every \( n \geq 2 \). In the classical setting, toroidal automorphic forms were introduced by Don Zagier in \cite{zagier-81}, where he showed that if the space of toroidal automorphic forms is a unitarizable representations, then a formula of Hecke implies the Riemann Hypothesis. These automorphic forms were further investigated by Lorscheid in \cite{oliver-toroidal}, in the context of global function fields.

\section{Hecke modifications }
 
In this section, we let \( X \) be a smooth projective curve defined over an arbitrary field \( k \). After introducing the main object of this article and its basic properties, we investigate the question of when Hecke modifications can be obtained from lower-rank Hecke modifications.

\subsection*{Preliminaries}

 Let $x \in X$ be a closed point of degree $|x|$ and residue field $\kappa(x)$. Let $\pi_x$ be a uniformizer for $x$. Let $ r \in \Z$ be a positive integer, we denote by $\cK_{x}^{\oplus r}$ the skyscraper sheaf supported at $x$ with stalk $\kappa(x)^{\oplus r}$. Let $\mathrm{Coh} (X)$ be the category of coherent sheaves on $X$ and $\Bun_n (X)$ be the set of isomorphism classes of rank $n$ vector bundles on $X$.  We shall consider vector bundle on $X$ as locally free sheaves on $X$. Hence, we consider $\Bun_n (X)$ to be embedded in  $\Coh (X)$, cf.\  \cite[Ex. II.5.18]{hartshorne}. We denote $\Bun_1 (X)$ by $\Pic (X)$, which is an abelian group with the group operation given by the tensor product. We let $\mathcal{O}$ stand for the structural sheaf of $X$. 

\begin{df} \label{def-heckemod}
    Given a vector bundle $\E \in \Bun_n (X)$, a closed point $x \in X$, and a positive integer $r \in \Z$. We denote by $[\E' \xrightarrow[r]{x} \E]$ the isomorphism class of exact sequences
    \[ 0 \rightarrow \E' \rightarrow \E \rightarrow \cK_x^{\oplus r} \rightarrow 0, \]
    where two such sequences
    \[ 0 \rightarrow \E_1 \rightarrow \E \rightarrow \cK_x^{\oplus r} \rightarrow 0 \ \ \text{and} \ \ 0 \rightarrow \E_2 \rightarrow \E \rightarrow \cK_x^{\oplus r} \rightarrow 0 \]
    are equivalent if there are isomorphisms $\E_1 \rightarrow \E_2$ and $\cK_x^{\oplus r} \rightarrow \cK^{\oplus r}_x$ such that the following diagram commutes:
    $$\xymatrix{ 0 \ar[r] & \E_1 \ar[r] \ar[d]^{\cong} & \E \ar[r] \ar@{=}[d] & \cK_x^{\oplus r} \ar[r] \ar[d]^{\cong} & 0 \\
    0 \ar[r] & \E_2 \ar[r] & \E \ar[r] & \cK_x^{\oplus r} \ar[r] & 0. }$$
    If such an exact sequence exists, we say that $\E'$ is a \textbf{Hecke modification of $\E$ at $x$ with weight $r$} or, alternatively, that $\E$ is Hecke modified in $\E'$ at $x$ with weight $r$. When $r=1$, the notation $[\E' \xrightarrow[r]{x} \E]$ is simplified to $[\E' \xrightarrow[]{x} \E]$.
\end{df}

\begin{df}
   If $\mathcal{F}$ and $\G$ are coherent sheaves on $X$, then $\Ext^1(\mathcal{F},\G)$ is a finite-dimensional vector space over $k$. Hence, when $k$ is a finite field, we may also define for $\E,\E' \in \Bun_n (X)$ the quantity $m_{x,r}(\E',\E)$ as the number of isomorphism classes of exact sequences
$$0 \longrightarrow \E'' \longrightarrow \E \longrightarrow \cK_{x}^{\oplus r} \longrightarrow 0$$
with fixed $\E$ such that $\E'' \cong \E'.$ We refer to $m_{x,r}(\E',\E)$ as the \textbf{multiplicity of the Hecke modification $[\E' \xrightarrow[r]{x} \E]$}. 
\end{df}

\begin{ex}
    Let $X$ be the projective line $\Proj(\FF_2[S,T])$ over $\FF_2$. Let $x$ be the closed point of degree $2$  corresponding to the maximal ideal $\langle T^2+ST+S^2 \rangle$. There are five Hecke modifications of $\E = \cO \oplus \cO$ at $x$ with weight $1$:  
    \begin{itemize}
        \item $\cO(-1)\oplus \cO(-1)$ with morphism classes 
    \[\varphi_1(U) =\begin{pmatrix}
        T & T+S \\
        S & T
    \end{pmatrix} \ \ \text{ and } \ \  \varphi_2(U) =\begin{pmatrix}
        S & T+S \\
        T & S
    \end{pmatrix}, \ \text{and} \]
    \item $\cO(-2)\oplus\cO$, with three morphism classes
\[\varphi_3 =\begin{pmatrix}
        T^2+ST+S^2 & 0\\
        0 & 1
    \end{pmatrix} , \ \ 
    \varphi_4 =\begin{pmatrix}
        T^2+ST+S^2 & 1\\
        0 & 1
    \end{pmatrix} \ \ \text{ and} \ \ \varphi_5 =\begin{pmatrix}
        T^2+ST+S^2 & 0\\
        T^2+ST+S^2 & 1
    \end{pmatrix}.\]
    \end{itemize}

    Therefore, we conclude that 
    \[m_{x,1}(\cO(-1)\oplus\cO(-1), \cO\oplus\cO )=2 \  \text{ and } \  m_{x,1}( \cO(-2)\oplus\cO, \cO\oplus\cO)=3. \]
    The sum of these multiplicities is equal to $\#\Gr(1,2)({\FF_{2^2}})=5$. This is not a coincidence; indeed, we can identify the set of Hecke modifications of $\E \in \Bun_n(X)$ at $x$ with weight $r$ with the $\kappa(x)$-rational points of the Grassmannian $\Gr(n-r,n)$, according to the  following theorem. 
\end{ex}

\begin{thm} \label{thm-grass}
     Let $\E \in \Bun_n (X)$. The set of Hecke modifications of $\E$ at $x$ with weight $r$ can be canonically identified with the set of dimension $n-r$ subspaces of the $\kappa(x)$-vector space $\E_x \otimes \kappa(x)$. In other words, 
    the set of Hecke modifications of $\E$ at $x$ with weight $r$ is the set of $\kappa(x)$-rational points of the Grassmannian $\mathrm{Gr}(n-r,n).$ 
\end{thm}

\begin{proof} See either \cite[Thm. 2.6]{alvarenga19} or \cite[Lemma 2.4]{dennis-03}.\end{proof}

\begin{rem} Given a Hecke modification $[\E' \xrightarrow[r]{x} \E]$, taking into account the stalk at $x$, we have a sequence of $\cO_{x}-$modules
\[ 0 \rightarrow \E_{x}' \rightarrow \E_{x} \rightarrow \kappa(x)^{\oplus r} \rightarrow 0.\]
This short exact sequence is no longer injective when evaluated on the geometric fiber at $x$ (i.e., after tensorising with $\kappa(x)$). Consequently, we may write the restriction on the fiber at $x$ as
\[ 0 \rightarrow \ker(\E_{x}' \otimes \kappa(x) \rightarrow \E_{x} \otimes \kappa(x) ) \rightarrow \E_{x}' \otimes \kappa(x) \rightarrow \E_{x} \otimes \kappa(x) \rightarrow \kappa(x)^{\oplus r} \rightarrow 0.\]
The previous construction associates $[\E' \xrightarrow[r]{x} \E]$ to a vector subspace of dimension $r$ in $\E_{x}' \otimes \kappa(x) $ and to a vector subspace of dimension $n-r$ in 
$\E_{x} \otimes \kappa(x)$. Therefore we can restrict the study of Hecke modifications to weights $r$ ranging from $0$ to $n=\rk(\E)$.  

Conversely, given a subspace $V \subset \E_{x}' \otimes \kappa(x)$ of dimension $r$, we define $\E$ to be the subsheaf of $\E'(x) := \E' \otimes \mathcal{O}_X(x)$ whose set of sections over an open set $U \subseteq X$ is given by
\[\E(U) := \big\{ s \in \E'(x)(U) \,\big|\, s= \pi_x^{-1}t, t \in \E'(U), \text{ and if } x \in U, \text{ then } t(x) \in V \big\}\]
where $\pi_x$ is a uniformizer of $x$ in $X$. Thus, we obtain 
\[0 \longrightarrow \E' \longrightarrow \E \longrightarrow \cK_{x}^{\oplus r} \longrightarrow 0, \]
a short exact sequence of coherent sheaves.

Moreover, given a subspace $W \subset \E_{x} \otimes \kappa(x)$ of dimension $n-r$, we define $\E'$ to be the subsheaf of $\E$ whose set of sections over an open set $U \subseteq X$ is given by
\[\E'(U) := \big\{ s \in \E(U) \,\big|\, \text{ if } x \in U, \text{ then } s(x) \in W \big\}.\]
Hence we also obtain 
\[0 \longrightarrow \E' \longrightarrow \E \longrightarrow \cK_{x}^{\oplus r} \longrightarrow 0 \]
a short exact sequence of coherent sheaves.
\end{rem}

\subsection*{Smith Normal Form} In the following, we apply the Elementary Divisor Theorem to show that a morphism $\varphi$ realizing a Hecke modification $[\E' \xrightarrow[r]{x} \E]$ can be represented, up to automorphism of $\E'$ and $\E$, as a diagonal matrix in a neighborhood $U$ of $x$. This diagonal form, known as Smith Normal Form of $\varphi(U)$, will be denoted by $\mathrm{SNF}(\varphi(U))$.  This local description will allow us to investigate some Hecke modifications from lower-rank Hecke modifications, which is the main goal of this section.

\begin{prop}
    Let  $\E, \E'\in \Bun_n(X)$ be such that $[\E' \xrightarrow[r]{x} \E]$. Let $\varphi:\E'\to \E$ be a morphism realizing $[\E' \xrightarrow[r]{x} \E]$. Then, there is $U$ a sufficiently small affine neighborhood of $x$ in $X$, such that  $\varphi(U)$ admits a Smith Normal Form 
    $$\mathrm{SNF}(\varphi(U))= \begin{pmatrix}
    \alpha_1 & 0 & 0 & \cdots & 0\\
    0 & \alpha_2 & 0 & \cdots & 0 \\
    \vdots & & & \cdots &\vdots \\
    0 & 0 & 0 & \cdots & \alpha_n
\end{pmatrix}=: \mathrm{diag}(\alpha_1, \dots , \alpha_n),$$
where \[ \alpha_i = \left\{\begin{array}{ll} 1, \ \ \ \  \text{if} \  i\leq n-r \\ \pi_x  \ \ \  \ \text{if} \  i>n-r.\end{array} \right. \]
\end{prop}

\begin{proof}
The Hecke modification $[\E' \xrightarrow[r]{x} \E]$ can be represented by a short exact sequence of coherent sheaves:
     \[      0 \longrightarrow \E' \overset{\varphi}{\longrightarrow}  \E  \longrightarrow \cK_x^{\oplus r} \longrightarrow 0.  
    \] 
Let $U$ be a sufficiently small affine neighborhood of $x$ such that $\E(U)$ and $\E'(U)$ are free modules over the principal ideal domain $\cO(U)$ and let $\pi_x $ be a uniformizer for $\cO_x$ in $\cO(U)$.

Evaluating the above exact sequence on $U$, yields a short exact sequence of $\cO(U)$-modules:
      \[      0 \longrightarrow \E'(U) \overset{\varphi(U)}{\longrightarrow}  \E(U)  \longrightarrow \kappa(x)^{\oplus r} \longrightarrow 0.\]
  Let $\mathrm{Mat}_n(\cO(U))$ be the $\cO(U)$-module consisting of square matrices of order $n$ with entries in $\cO(U)$.  By the Elementary Divisor Theorem, there are bases for $\E'(U)$ and $\E(U)$ such that the matrix $\varphi(U) \in \mathrm{Mat}_n(\cO(U))$ admits a Smith Normal Form $\mathrm{SNF}(\varphi(U))\in \mathrm{Mat}_n(\cO(U))$. Namely, 
    $$\mathrm{SNF}(\varphi(U))= \begin{pmatrix}
    \alpha_1 & 0 & 0 & \cdots & 0\\
    0 & \alpha_2 & 0 & \cdots & 0 \\
    \vdots & & & \cdots &\vdots \\
    0 & 0 & 0 & \cdots & \alpha_n
\end{pmatrix}=: \mathrm{diag}(\alpha_1, \dots , \alpha_n) \in \mathrm{Mat}_n(\cO(U)),$$
where $\alpha_i | \alpha_{i+1}$ for $i = 1, \dots , n-1$. Under these conditions
\[\kappa(x)^{\oplus r}=\mathrm{coker}(\varphi(U))= \dfrac{\E(U)}{\mathrm{Im}(\varphi(U))} \cong \dfrac{\cO(U)^n}{\langle \alpha_1, \dots , \alpha_n \rangle} \cong \dfrac{\cO(U)}{\langle \alpha_1 \rangle} \oplus \cdots \oplus \dfrac{\cO(U)}{\langle \alpha_n \rangle}.\]
This implies that, up to multiplication by invertible elements, the values of $\alpha_i$ are given by:
\[ \alpha_i = \left\{\begin{array}{ll} 1, \ \ \ \  \text{if} \  i\leq n-r \\ \pi_x  \ \ \  \ \text{if} \  i>n-r.\end{array} \right. \]
Which is the desired conclusion.
\end{proof}

\begin{rem}
    In the same conditions of above proposition, the elements $\alpha_i \in \cO(U)$, known as elementary divisors of $\varphi(U)$, are uniquely determined by
\[\alpha_1 \alpha_2 \cdots \alpha_i = \gcd(\text{minors } i\times i \text{ of } \varphi(U)),\]
cf. \cite[Thm 2.4]{SNF}
\end{rem}

\begin{rem}\label{blocks}
    Let $\E, \E' \in \Bun_n(X)$ and $[\E' \xrightarrow[r]{x} \E]$ be a Hecke modification. Suppose that there exist $\E_1', \E_1 \in \Bun_{n_1}(X)$ and $\E_2', \E_2 \in \Bun_{n_2}(X)$ such that $$\E'= \E_1' \oplus \E_2' \ \ \text{and} \ \  \E=\E_1 \oplus \E_2.$$
    Let $\Phi$ denote the set of morphisms $\varphi: \E' \to \E$ satisfying $\mathrm{coker}(\varphi)=\cK_x^{\oplus r}$. We can express an element $\varphi \in \Phi$ in block matrix form as follows:
    \[\varphi = \begin{pmatrix}
        A_{n_1 \times n_1}: \E_1' \to \E_1 & B_{n_1 \times n_2}: \E_2' \to \E_1 \\
        C_{n_2 \times n_1}: \E_1' \to \E_2 & D_{n_2 \times n_2}: \E_2' \to \E_2
    \end{pmatrix}.\]
\end{rem}

\begin{prop}\label{break0}
    Let $\E', \E' \in \Bun_n(X)$ be such that 
    \[ \E=\E_1 \oplus \E_2  \quad \text{and} \quad  \E'=\E_1' \oplus \E_2',\]
    with $\rk(\E_1)= \rk(\E_1')=n_1$ and $\rk(\E_2)= \rk(\E_2')=n_2$. Let $[\E_1'\oplus\E_2' \xrightarrow[r]{x} \E_1 \oplus \E_2]$ be a Hecke modification. Then there exist integers $r_1 \in \{0, 1, \dots, r\}$ and $r_2=r-r_1$ such that $ [\E_1' \xrightarrow[r_1]{x} \E_1]$ and $[\E_2' \xrightarrow[r_2]{x} \E_2]$  are Hecke modifications if and only if (in the notation of Remark \ref{blocks}) there exists $\varphi \in \Phi$ such that, in the block matrix representation,   $B=0_{n_1 \times n_2}$.
\end{prop}

\begin{proof}
Suppose that $[\E_1' \xrightarrow[r_1]{x} \E_1]$ and $[\E_2' \xrightarrow[r_2]{x} \E_2]$ are Hecke modifications given by morphisms $\varphi_i:\E_i' \to \E_i$, $i=1,2$. Define $\varphi:\E'\to \E$ given by the following matrix representation in blocks 
\[\varphi = \begin{pmatrix}
        \varphi_1 & 0 \\
        0 & \varphi_2
    \end{pmatrix}.\]

Let $U$ be any sufficiently small affine neighborhood of $x$ such that $\E(U),\E_1(U),\\ \E_2(U), \E'(U), \E'_1(U)$ and $\E'_2(U)$ are free modules over a principal ideal domain. Hence, any $i \times i$ minor of the matrix $\varphi(U)$ with nonzero determinant can be represented as
\[\varphi(U) = \begin{pmatrix}
        M_{i_1 \times i_1} & 0_{i_1 \times i_2} \\
        0_{i_2 \times i_1} & N_{i_2 \times i_2}
    \end{pmatrix},\]
where $i=i_1+i_2$, $M$ is a minor $i_1 \times i_1$ of $\varphi_1(U)$ and $N$ a minor $i_2 \times i_2$ of $\varphi_2(U)$. 

Since
\[\mathrm{SNF}(\varphi_1(U))=\mathrm{diag}(1, \dots, 1, \underbrace{\pi_x, \ldots, \pi_x}_{r_1 \text{ times }})\] 
and 
\[ \mathrm{SNF}(\varphi_2(U))=\mathrm{diag}(1, \dots, 1, \underbrace{\pi_x, \ldots, \pi_x}_{r_2 \text{ times }}),\]
we conclude that 
\[\mathrm{SNF}(\varphi(U))=\mathrm{diag}(1, \dots, 1, \underbrace{\pi_x, \ldots, \pi_x}_{r_1+r_2 \text{ times }}). \]
This implies that $\mathrm{coker}(\varphi)=\cK_x^{\oplus r}$, which yields a Hecke modification $[\E' \xrightarrow[r]{x} \E]$.

Next, let  $[\E' \xrightarrow[r]{x} \E]$ be a Hecke modification defined by the morphism $\varphi: \E' \to \E$.  Through a block matrix representation, as described in Remark \ref{blocks}, $B=0$ and $n_i$ is the rank of $\E_i$, $i=1,2$. Thus, up to associates,
\[\det(\varphi(U))=\det (\mathrm{SNF}(\varphi(U)))= \pi_x^r = \det(A(U)) \det(D(U)).\]
Since $\pi_x$ is irreducible, there are non-negative integers $r_1,r_2 \in \Z$, $r_1 \leq r$ and $r_2=r-r_1$, such that $\det(A(U))= \pi_x^{r_1}$ and $\det(D(U))=\pi_x^{r_2}$.
Thus \[\mathrm{SNF}(\varphi(U))=\mathrm{diag}(1, \dots, 1, \underbrace{\pi_x, \ldots, \pi_x}_{r \text{ times }}).\]
Hence, for $i\in \{0, \dots, r\}$, the $(n-i)$-th elementary divisor of $\varphi(U)$, 
\begin{equation}\label{determinantal}
    \alpha_{n-i}=\gcd(\{ \text{ minors } (n-i)\times (n-i) \text{ of } \varphi(U) \}) = \pi_x^{r-i}.
\end{equation}

Let $\alpha^A_{i_1}$ and $\alpha^D_{i_2}$ denote the elementary divisors of $A(U)$ and $D(U)$, respectively, with $i_1 \in \{1, \dots, n_1\}$ and $i_2 \in \{1, \dots, n_2\}$.
Given $i_1 \in \{0, \dots, r_1\}$, let $M_A$  be a minor $(n_1-i_1) \times (n_1-i_1)$ of the matrix $A(U)$. Let $M$ be a minor $(n-i_1)\times (n-i_1)$ of $\varphi(U)$ in the form
\[M = \begin{pmatrix}
        M_A & 0 \\
        \Tilde{C}(U) & D(U)
    \end{pmatrix}.\]
Here, $\Tilde{C}(U)$ represents a submatrix of $C(U)$ related with the choice of rows of $M_A$. The determinant of $M$ is given by $\det(M)=\det(M_A) \det(D(U))=\pi_x^{r_2} \det(M_A)$. Consequently, identity (\ref{determinantal}) yields $\alpha_{n-i_1}=\pi_x^{r-i_1}$, which must divides $\det(M)$. Hence, 
\[\pi_x^{r-i_1} \mid \pi_x^{r_2} \det(M_A) \ \ \ \text{meaning} \ \  \pi_x^{r_1-i_1}\mid\det(M_A).\]

Since $\det(A(U))\neq 0$, for all $i_1 \in \{1, \dots, r_1\}$ there is a minor $M_A$ with determinant non-zero. Then, using the formula of elementary divisor, we conclude that, except by associates
\[\alpha_{n_1-i_1}^A=\left\{\begin{array}{ll} \pi_x^{r_1-i_1}, &  \text{if } i_1 \in \{0, \dots , r_1\} \\ 
1,  &  \text{if} \  i_1 \in \{r_1+1 \dots , n_1-1\}.\end{array} \right.  \]
Therefore, $A$ represents a morphism $\varphi_1:\E_1' \to \E_1$ with $\mathrm{coker}(\varphi_1)=\cK_X^{\oplus r_1}$, realizing a Hecke modification $[\E_1' \xrightarrow[r_1]{x} \E_1]$. Analogously, we obtain $[\E_2' \xrightarrow[r_2]{x} \E_2]$.
\end{proof}

\begin{lemma}\label{semistzero}
    Let $\cE$ and $\cF$ be semistable vector bundles on a curve $X$. If $\mu(\cE)>\mu(\cF)$, then $\Hom(\cE, \cF)=0$.
\end{lemma}

\begin{proof} See \cite[Prop 5.3.3]{lepotier}. \end{proof}

\begin{thm}\label{breaksemistable} 
 Let $x \in X$ be a closed point and $\E, \E'\in \Bun_n(X)$. Write 
 $ \E:=\bigoplus_{i=1}^{m_1}\F_i$ and  $\E':=\bigoplus_{i=1}^{m_2}\F'_i$
 where $\cF_i$ and $\cF'_i$ are semistable bundles with $\mu(\cF_i) \leq \mu(\cF_{i+1})$ and $\mu(\cF_j ') \leq \mu(\cF_{j+1}')$ for $i= 1, \ldots, m_1$ and $j=1, \ldots, m_2$. Suppose that there exist $a,b \in \Z_{>0}$ such that
\[\E_1=\bigoplus_{i=1}^{a-1}\cF_i, \ \ \E_2=\bigoplus_{i=a}^{m_1}\cF_i, \ \ \E_1'=\bigoplus_{i=1}^{b-1}\cF_i'\ \ \text{and} \ \ \E_2'=\bigoplus_{i=b}^{m_2}\cF_i'\]
with $\rk(\E_1')=\rk(\E_1)$ and
$\mu(\cF_j')>\mu(\cF_i)$ for $j = b, \dots, m_2$ and $i = 1, \dots, a-1$.
Then $[\E' \xrightarrow[r]{x} \E]$ is a Hecke modification if and only if $[\E_1' \xrightarrow[r_1]{x} \E_1]$ and $[\E_2' \xrightarrow[r_2]{x} \E_2]$ are Hecke modifications, where $r_1 :=(\deg(\E_1')-\deg(\E_1))/|x|$ and $r_2 :=r-r_1$.
\end{thm}

\begin{proof} Let $\varphi: \E' \to \E$ be a morphism realizing $[\E' \xrightarrow[r]{x} \E]$. We might write $\varphi = (\varphi_{ij})_{ij}$, where $\varphi_{ij} \in \Hom(\cF_j', \cF_i)$. According to Lemma \ref{semistzero}, $\Hom(\cF_{j}',\cF_i) = \{0\}$ for all $j= b, \dots, m_2$ and $i = 1, \dots, a-1$. Thus, $\varphi$ might be written as
   \[\varphi = \begin{pmatrix}
        A: \E_1' \to \E_1 & 0: \E_2' \to \E_1 \\
        C: \E_1' \to \E_2 & D: \E_2' \to \E_2
    \end{pmatrix},\]
and the theorem follows from Theorem \ref{break0}.
\end{proof}

\section{Sums of line bundles}

In this section, we explore the Hecke modifications of vector bundles given as sums of line bundles. As in the previous section, we let \( X \) be a smooth projective curve defined over an arbitrary field \( k \). 

\begin{rem}\label{padrao} Let  $\E, \E' \in \Bun_n(X)$ be isomorphic to the direct sum of line bundles i.e., 
\[  \E \cong \bigoplus_{i=1}^n\cL_i 
\quad \text{and} \quad
\E' \cong \bigoplus_{i=1}^n\cL_i'
\]
with $\cL_i, \cL_i' \in \Pic X$,  $\deg(\cL_i) \leq \deg(\cL_{i+1})$ and $\deg(\cL_i') \leq \deg(\cL_{i+1}')$, $i=1, \dots , n-1$. If $[\E' \xrightarrow[r]{x} \E]$, follows from \cite[Prop. 4.12]{alvarenga19} that we can consider $\deg(\cL_i') \leq \deg(\cL_i)$, for $i=1, \ldots,n$. 
In what follows, we will consistently work within this setting when considering Hecke modifications of direct sums of line bundles. Moreover, we denote $\epsilon_i := \deg(\cL_i) - \deg(\cL_i')$ for each $i \in \{1, \ldots, n\}$. 
\end{rem}

\begin{lemma}\label{dual}
    Let $\E, \E' \in \Bun_n(X)$. Then for every
short exact  sequence of the form
 \[0 \longrightarrow \E' \longrightarrow \E \longrightarrow  \cK_x^{\oplus r} \longrightarrow 0, \]
there exists a canonical short exact sequence
 \[0 \longrightarrow \E \longrightarrow \E'(x) \longrightarrow  \cK_x^{\oplus n- r} \longrightarrow 0.\]
where $\E'(x) := \E' \otimes_{\cO} \cO(x)$.
\end{lemma}

\begin{proof}
    See \cite[Lemma 2.1]{rov-elliptic1}.
\end{proof}

\begin{thm}\label{menorqued} Let
     \[0 \longrightarrow \cL_1' \oplus \cdots \oplus \cL_n '\longrightarrow \cL
_1 \oplus \cdots \oplus \cL_n \longrightarrow \cK_x^{\oplus r} \longrightarrow 0\]
be a Hecke modification as in Remark \ref{padrao}. Then $0 \leq \epsilon_i \leq |x|$, for all $i \in \{1, \ldots , n\}$.
\end{thm}

\begin{proof}
    From Lemma $\ref{dual}$, we might consider the dual exact sequence
\[ 0 \longrightarrow \cL_1 \oplus \cdots \oplus \cL_n \longrightarrow \cL_1'(x) \oplus \cdots \oplus \cL_n'(x) \longrightarrow  \cK_x^{\oplus n-r} \longrightarrow 0. \]
Let $d :=|x|$ and $d_i :=\deg(\cL_i)$, with $i \in \{1, \ldots, n\}$.
 
 Suppose that $\epsilon_j > d $ for some $j \in \{1, \ldots, n\}$. Suppose moreover that $j$ is the smaller index such that $\epsilon_j > d $. Thus 
 \[ d_n \geq \cdots \geq d_{j+1} \geq d_j >  d_j-\epsilon_j + d = \deg(\cL_j'(x)).\]
The injectivity of the dual exact sequence implies, by \cite[Prop. 4.12]{alvarenga19},  that there is an index $\ell_1$, with $\ell_1 < j $ such that $d_j-\epsilon_j+d \geq d_{\ell_1}$. By the same reason there is a index $\ell_2 \neq \ell_1$ such that $\deg(\cL_{\ell_1}'(x))=d_{\ell_1}-\epsilon_{\ell_1}+d \geq d_{\ell_2}$.

If $\ell_2 \geq j$, 
\[ d_{\ell_1}-\epsilon_{\ell_1}+d \geq d_{\ell_2} \geq d_{\ell_2-1} \geq \cdots \geq d_{j} > d_j-\epsilon_j+d \geq d_{\ell_1}.\]
This means that $ d_{\ell_1}-\epsilon_{\ell_1}+d > d_j-\epsilon_j+d$. Hence $\deg(\cL_{\ell_1})> \deg(\cL_j)$, which is a contradiction since $\ell_1<j$. Therefore $\ell_2<j$. 

By repeatedly applying the same argument $j-1$ times, we obtain a set of distinct indices $J:= \{\ell_0=j, \ell_1, \ldots , \ell_{j-1} \}$, with
$d_{\ell_i}-\epsilon_{\ell_i}+d \geq d_{\ell_{i+1}}$, for $i \in \{0, \ldots j-2\}$. However, in order to proceed with the construction, the index $\ell_{j-1}$ must be compared with some index $\ell<j$ such that $d_{\ell_{j-1}}-\epsilon_{\ell_{j-1}}+d \geq d_{\ell},$ which is impossible since $\#J=j$. Therefore,  $\epsilon_j \leq d$.
\end{proof}


 \begin{cor}\label{break}
Let $ \E=\bigoplus_{i=1}^n\cL_i$ and $ \E'=\bigoplus_{i=1}^n\cL'_i$  as in Remark \ref{padrao}. Suppose that there exists an index $j$ such that $\deg(\cL_j)-\deg(\cL_{j-1})>|x|$. Let
\[\E_1=\bigoplus_{i=1}^{j-1}\cL_i, \ \ \E_2=\bigoplus_{i=j}^{n}\cL_i, \ \ \E_1'=\bigoplus_{i=1}^{j-1}\cL_i'\ \ \text{and} \ \ \E_2'=\bigoplus_{i=j}^{n}\cL_i',\]
$r_1=\dfrac{\deg(\E_1')-\deg(\E_1)}{|x|}$ and $r_2=r-r_1$. Then $[\E' \xrightarrow[r]{x} \E]$ if and only if $[\E_1' \xrightarrow[r_1]{x} \E_1]$ and $[\E_2' \xrightarrow[r_2]{x} \E_2]$.
\end{cor}
\begin{proof}
    By the previous theorem 
    \[\mu(\cL_i)=\deg(\cL_i')< \deg(\cL_{\ell})=\mu(\cL_{\ell}) \ \ \text{ for $1 \leq i \leq j-1$ and $j \leq \ell \leq n$.}\]
    Therefore, the corollary follows directly from Theorem $\ref{breaksemistable}$.
\end{proof}

\begin{df}
     Let $n$ be a positive integer. For each $r \in \{0, 1, \ldots, n\}$, we define $$\Delta_r^n := \{f:\{1, 2, \ldots , n\} \to \{0,1\}\,\big|\, \  \# \supp(f)=r\}.$$ 
Moreover, for $\delta \in \Delta_r^n$, let
\[|\delta| := \sum_{i=1}^n(1-\delta(i))(r-\sum_{j=1}^i\delta(j)).\]
 \end{df} 

 \begin{ex} \label{funvec}
    An element $\delta \in \Delta_r^n$ can be seen as a vector in $\{0,1\}^n$ with exactly $r$ entries equal to $1$. Hence, we might interpreted $|\delta|$ as the sum over the entries equal to $0$ that adds $r$ minus the number of previous entries equal to $1$. As an example, if $n=6$ and $r=2$:
   \[ \begin{aligned}
        \delta_1= & (0,1,1,0,0,0) \in \Delta_2^6 & & & |\delta_1|= & \ 2+0+0+0+0+0 =2 \\
       \delta_2= & (1,0,0,0,1,0) \in \Delta_2^6 & & & |\delta_2| = & \ 0+ 1 + 1 + 1 + 0 + 0 =3 
    \end{aligned}\]
\end{ex}

\begin{lemma} \label{existtrivial}
    Let $\cE = \bigoplus_{i=1}^n \cL_i$ with $\cL_i \in \Pic X$ for $i=1, \ldots, n$. Given $r \in \{1, \ldots, n\}$ and $\delta \in \Delta_r^n$, let 
    \[ \E' := \bigoplus_{i=1}^n \cL_i \otimes \cO(-\delta(i)x).\] 
    Then there exists a $\varphi:\E'\to \E$ such that $[\E' \xrightarrow[r]{x} \E]$.
\end{lemma}
\begin{proof}
    Consider the exact sequence
    \[0 \rightarrow \cO(-x) \rightarrow \cO \rightarrow \cK_x \rightarrow 0.\]
    The functor $\cL_i \otimes - $ is exact, which implies that for each $i \in \{1, \ldots, n\}$, there are exact sequences
      \[0 \rightarrow \cL_i(-x) \rightarrow \cL_i \rightarrow \cK_x \rightarrow 0,\]
      and 
        \[0 \rightarrow \cL_i \rightarrow \cL_i \rightarrow 0 \rightarrow 0.\]
        Combining those sequences yields
          \[0 \rightarrow \bigoplus_{i=1}^n \cL_i \otimes \cO(-\delta(i)x) \rightarrow\bigoplus_{i=1}^n \cL_i \rightarrow \cK_x^{\oplus r} \rightarrow 0,\]
          a short exact sequence. This completes the proof.
\end{proof}
 
\begin{cor}
  Let $\cE \in \Bun_n(X)$ isomorphic to a sum of line bundles $\cE= \bigoplus_{i=1}^n \cL_i$. Suppose $\deg(\cL_{i+1})-\deg(\cL_i) > |x|$ for all $i \in \{1, \ldots, n-1\}$. Fix an $r \in \{1, \ldots, n\} $. Then,  $ \E'= \bigoplus_{i=1}^n \cL_i'$, with $\cL_i'\in \Pic(X)$  for $i=1, \ldots, n$, is a Hecke modification of $\E$ at $x$ with weight $r$ if and only if there exists $\delta \in \Delta_r^n$ such that 
   \[ \cL_i' \cong \cL_i \otimes \cO(-\delta(i)x) \]
   for each $ i \in \{1, \ldots, n\}.$
\end{cor}

\begin{proof}
    By Corollary \ref{existtrivial}, for all $\delta \in \Delta_r^n$, $ \E'= \bigoplus_{i=1}^n \cL_i \otimes \cO(-\delta(i)x)$ is a Hecke modification of $\E$ at $x$ with weight $r$. Conversely, let $\cF=\bigoplus_{i=1}^n\cL_i'' \in \Bun_n(X)$, with $\cL_i''\in \Pic(X)$, such that $[\cF \xrightarrow[r]{x} \E]$. By Theorem \ref{menorqued}
    \[\deg(\cL_i'')<\deg(\cL_i) \ \ \text{ for all $i=1, \dots, n$}.\]
    Thus, Corollary \ref{break} yields $[\cL_i'' \xrightarrow[r_i]{x} \cL_i]$ for all $i \in \{1, \dots, n\}$ with $r_i=0$ or $1$. Those weights are equal to $1$ for exactly $r$ indices. This allows us to define $\delta \in \Delta_r^n$,  by setting $\delta(i):=r_i$  for each $i \in {1, \ldots, n}$. Therefore, 
\[\cF= \bigoplus_{i=1}^n \cL_i \otimes \cO(-\delta(i)x),\]
which completes the proof.
\end{proof}

We finish this section with an application of the previous discussion to the case where $X$ is the projective line.
\begin{thm}[Birkhoff-Grothendieck]\label{birgro}
    Every rank $n$ vector bundle over $\mathbb{P}^1$ is isomorphic to
$$\mathcal{O}_{\mathbb{P}^1}(d_1) \oplus \cdots \oplus \mathcal{O}_{\mathbb{P}^1}(d_n)$$
for some integers $d_1 \leq \cdots \leq d_n.$ In particular, the line bundles are the unique indecomposable objects in the category of vector bundles over $\P^1$.  
\end{thm}

\begin{cor}\label{greater0}
    Let $\E = \bigoplus_{i=1}^n \cO(d_i) \in \Bun_n (\P^1)$ be as Theorem \ref{birgro}. Then, every Hecke modification of $\E$ at $x$ with weight $r$ can be represented in the form
\begin{equation}\label{odhec}
 0 \longrightarrow \bigoplus_{i=1}^n \cO(d_i-\epsilon_i) \longrightarrow \bigoplus_{i=1}^n \cO(d_i) \longrightarrow  \cK_x^{\oplus r} \longrightarrow 0, 
    \end{equation}
for some $\epsilon_i \in \Z$, with $0 \leq \epsilon_i \leq |x|$ for all $i=1, \ldots, n$.   
\end{cor}

\begin{proof}
Let $\E' \in \Bun_n (\P^1)$ be a Hecke modification of $\E$.  
    By Theorem \ref{birgro}, we can write $ \E' \cong \bigoplus_{i=1}^n \cO(d_i')$ for some integers $d_i'$ with $d_i' \leq d_{i+1}'$, for $i=1, \ldots, n-1$. By Theorem \ref{menorqued}, $d_i' = d_i-\epsilon_i$, for some $0 \leq \epsilon_i \leq |x|$, which concludes the proof.
\end{proof}

\begin{thm} \label{thm-heckemodforP1anyfield}
    Let $k$ be an algebraically closed field. Let $\E, \E' \in \Bun_n (\P^1)$. Write $\E = \bigoplus_{i=1}^n \cO(d_i), \E' = \bigoplus_{i=1}^n \cO(d_i')$ as in Theorem \ref{birgro}. Then $\E'$ is a Hecke modification of $\E$ at $x$ with weight $r$ if and only if there is $\delta \in \Delta_r^n$ such that 
    \[ d_i ' = d_i -\delta (i).\]
\end{thm}

\begin{proof}
   Over an algebraically closed field, every closed point has degree $1$. Hence, Theorem $(\ref{menorqued})$ yields $\epsilon_i=\deg(\cL_i)-\deg(\cL_i)' \in \{0,1\}$. By additivity of the degree in short exact sequences $\epsilon_i=1$ for exactly $r$ index $i \in \{1, \ldots, n\}$.
\end{proof}

\section{The Hall Algebra}
 In this section, we always assume that $k = \FF_q$ is the finite field with $q$ elements, and let $X$ be the projective line over $\FF_q$. We apply the Hall algebra of $\Coh(\P^1)$ to provide structure results for the Hecke modifications (and their multiplicities) for vector bundles over $\P^1$.

In the following theorem, we summarize some basic properties of $\Coh (\P^1)$. 

\begin{thm}\label{hallalgebra}
    The category $\Coh (\P^1)$ is $\FF_q$- linear, abelian and  satisfies the following finiteness 
conditions:
\begin{enumerate}[(i)]
    \item The isomorphism classes of objects in $\Coh (\P^1)$ form a set $\Iso(\Coh (\P^1))$.
    \item For all objects $\cF, \cG$ in $\Coh (\P^1)$, the $\FF_q$-vector space $\Hom (\cF, \cG)$ is finite-dimensional.
     \item For all objects $\cF, \cG$ in $\Coh (\P^1)$, the $\FF_q$-vector space $\Ext^1 (\cF, \cG)$ is finite-dimensional.
     \item The category $\Coh (\P^1)$ can be embedded as a full subcategory in an abelian $\FF_q$-linear category $\cA$ with enough injectives (or projectives). Moreover, $\Coh (\P^1)$ is closed under extensions in $\cA$, and $\Ext^2_{\cA}(\cF, \cG) = 0$ for all objects $\cF, \cG$ in $\Coh (\P^1)$. 
     \item Each object in $\Coh(\P^1)$ has a finite filtration with simple quotients (Jordan-Hölder series).
\end{enumerate}
\end{thm}

\begin{proof} See \cite[Prop 3]{baumann-kassel-01} \end{proof}

We represent the class of an object $\alpha \in ob(\Coh (\P^1))$ in $\Iso(\Coh (\P^1))$ by itself, that is, $\overline{\alpha}=\alpha$. Given three isomorphism classes $\alpha, \beta, \gamma \in \Iso(\Coh (\P^1))$, we denote by $\phi_{\alpha \gamma}^{\beta}$ the number of sub-objects $\cF \in \Iso(\beta)$ such that $\cF \cong \gamma$ and $\beta/\cF \cong \alpha$. That is
\[\phi_{\alpha \gamma}^{\beta} = \dfrac{\#\{0\longrightarrow \gamma \longrightarrow \beta \longrightarrow \alpha \longrightarrow 0\}}{\#\Aut(\alpha)\# \Aut(\gamma)}.\]
The integer $\phi_{\alpha \gamma}^{\beta}$ is called the \textbf{Hall number of $(\alpha, \beta, \gamma)$}.

Since $\Ext^1 (\alpha, \gamma )$ is a finite set, there are only finitely many isomorphism classes $\beta$ such that $\phi^{\beta}_{\alpha \gamma} \neq 0$.


\begin{df}
    Let $\Tilde{Z}=\Z[v,v^{-1}]/(v^2-q)$. Let $H(\Coh (\P^1))$ be the free $\Tilde{Z}$-module on the symbols $\alpha \in \Iso(\Coh (\P^1))$. We can define a product in $H(\Coh (\P^1))$ as follows,
\[\alpha * \gamma = \sum_{\beta \in \Iso(\Coh (\P^1))} \phi_{\alpha \gamma}^{\beta} \beta.\]
With this product, $H(\Coh (\P^1))$ has the structure of an associative $\Tilde{Z}$-algebra with unit given by the zero element. This algebra is called the \textbf{Hall algebra of the category $\Coh (\P^1)$}.
\end{df}


\begin{rem}     The Hall algebra might be defined over any finitary category, see \cite{olivier-12}. In particular, the above definition holds for every smooth projective curve defined over $\FF_q$. Moreover, one can define the Hall algebra of a curve defined over an arbitrary field, see \cite{lusztig-91}. In this more general setting, the Hall number $h_{\F,\G}^{\mathcal{H}}$ is replaced by the Euler characteristic of the constructible space of all such objects (i.e., the space of all subobjects of $\cH$ of type $\cG$ and cotype $\F$). This variant of the Hall algebra is known as \emph{$\chi$-Hall algebra}.
\end{rem}


\begin{rem}
Let $r, n \in \Z_{>0}$ with $r \leq n$. Let $\alpha=\cK_x^{\oplus r}$ and $\gamma=\E' \in \Bun_n\P^1$. If $\E \in \Bun_n(\P^1)$ is such that $\phi^{\E}_{\cK_x^{\oplus r}\E'} \neq 0$ then $\E'$ is a Hecke modification of $\E$ at $x$ with weight $r$. Moreover,
$$m_{x,r}(\E', \E)=\phi^{\E}_{\cK_x^{\oplus r}\E'},$$
cf.\ \cite[Lemma 2.1]{alvarenga20}.
\end{rem}


\noindent\textbf{Notation.} We denote the Hall product of $n$ line bundles $\cO(d_1), \dots, \cO(d_n) \in \Pic(\P^1)$ in $H(\Coh (\P^1))$ by
\[\cO(d_1)* \cdots * \cO(d_n)=: \hprod_{i=1}^n \cO(d_i).\]
    

\begin{thm} \label{hallprop}
    In the Hall algebra $H(\Coh (\P^1))$ we have the following relations:
    \begin{enumerate}[(i)]
        \item If $\cF, \cG \in \Coh (\P^1)$ are such  that $\Hom(\cG,\cF)=0$ then $\cF*\cG=\cF \oplus \cG$.
        
        \item $ \cO(m)^{\oplus a} * \cO(m)^{\oplus b}= \left( \prod_{i=0}^{a-1} \dfrac{q^{a+b-i} -1}{q^{a-i}-1}\right) \cO(m)^{\oplus (a+b)}$ for every $m \in \Z$ \\ and $a,b \in \N$.
        
        \item $\displaystyle \hprod_{j=1}^{a}\cO(m) = \left( \prod_{i=0}^{a-1} \dfrac{q^{a-i} -1}{q-1}\right) \cO(m)^{\oplus a}$ for every $m \in \Z$ and $a \in \N$.
        
        \item If $m<n$, then \[\cO(n)*\cO(m)=q^{n-m+1}\cO(m)\oplus \cO(n)+\sum_{i=1}^{\lfloor \frac{n-m}{2} \rfloor}(q^2-1)q^{n-m-1}\cO(m+i)\oplus \cO(n-i).\]
        
        \item $\cK_x^{\oplus r}*\cO(m)= \cO(m+|x|) \oplus \cK_x^{r-1} +q^{r|x|}\cO(m)\oplus \cK_x^{r}.$
    \end{enumerate}
\end{thm}
\begin{proof}
     Item $(i)$ is a consequence of Serre's duality for curves. Items $(ii), (iv)$ and $(v)$ are in \cite[Thm. 13]{baumann-kassel-01}. Item $(iii)$ can be obtained applying induction in $(ii)$. 
\end{proof}


\begin{rem}
    Special cases where $(i)$ occurs are when: $\cF$ is a locally-free sheaf and $\cG$ is a torsion sheaf; $\cF$ and $\cG$ are torsion sheaves with disjoint support; and when $ \cF=\bigoplus_{i=1}^s\cO(m_i)$ and $\cG=\cO(n)^{\oplus a}$ with $m_i<n$ for all $i=1, \ldots , s$.
\end{rem}


As in Theorem \ref{birgro}, if $\E \in \Bun_n (\P^1)$, then 
\[ \E \cong \cO(d_1) \oplus \cdots \oplus \cO(d_n)\]
for some integers $d_1 \leq  \cdots \leq d_n$. In what follows, all vector bundles over $\P^1$ are given as above, i.e.\ the degree of its line bundles decomposition increase as the indices increase.

The following proposition illustrates how we can apply the Hall products to identify the Hecke modifications and their multiplicities.

\begin{prop}  Let $x \in \P^1$ be a closed point of degree $d$ and let $\E \in \Bun_2\P^1$ be represented as $\E \cong \cO(d_1) \oplus \cO(d_2)$ with $d1 \leq d_2$. Let $\ell := \lfloor \frac{d-d_2+d_1-1}{2} \rfloor$. Then the Hecke modifications of $\E$ at $x$, and their multiplicities, are given as follows: 

\noindent     
If $d_2-d_1 \geq d$,
  \[ m_{x,1}(\E, \E')= \begin{cases}
      q^d & \text{if } \E'\cong \cO(d_1) \oplus \cO(d_2-d); \\
        1 & \text{if } \E'\cong \cO(d_1-d)\oplus \cO(d_2).
    \end{cases}\]

\noindent     
If $0<d_2-d_1 < d$ and $d_2+d-d_1$ is even,
      \[ m_{x,1}(\E, \E')= 
      \begin{cases}
      q^d-q^{d-1} & \text{if } \E'\cong \cO (\frac{d_2+d_1-d}{2}) \oplus \cO(\frac{d_2+d_1-d}{2}); \\
        q^{d_2-d_1+1} & \text{if } \E'\cong \cO(d_1)\oplus \cO(d_2-d); \\
        1 & \text{if } \E'\cong \cO(d_1-d)\oplus \cO(d_2); \\
        q^{2i+2} - q^{2i} & \text{if } \E'\cong \cO(d_1-i)\oplus \cO(d_2-d+i)  \text{ with $i=1, \ldots, \ell$.}
    \end{cases}\]

\noindent     
If  $0<d_2-d_1 < d$ and $d_2+d-d_1$ is odd,
    \[ m_{x,1}(\E, \E')= \begin{cases}
        q^{d_2-d_1+1} & \text{ if } \E'\cong \cO(d_1)\oplus \cO(d_2-d); \\
        1 & \text{if } \E'\cong \cO(d_1-d)\oplus \cO(d_2); \\
        q^{2i+2} - q^{2i} & \text{if } \E'\cong \cO(d_1-i)\oplus \cO(d_2-d+i) \ \text{ with $i=1, \ldots, \ell$.}
    \end{cases}\]

\noindent     
If $d_1=d_2$ and $d$ is even,
        \[m_{x,1}(\E, \E')= \begin{cases}
           q^d-q^{d-1} & \text{if } \E'\cong \cO (d_1 - \frac{d}{2}) \oplus \cO(d_1 - \frac{d}{2});\\
            q+1  & \text{if }  \E'\cong \cO(d_1-d)\oplus \cO(d_1);\\ 
            q^{2i+1} - q^{2i-1} & \text{if } \  \E'\cong \cO(d_1-d+i)\oplus \cO(d_1-i)  \text{ with $i=1, \ldots, \lfloor \frac{d-1}{2} \rfloor$.} 
        \end{cases} \]

\noindent     
If $d_1=d_2$ and $d$ is odd,
        \[m_{x,1}(\E, \E')= \begin{cases}
            q+1 & \text{if }  \E'\cong \cO(d_1-d)\oplus \cO(d_1);\\ 
            q^{2i+1} - q^{2i-1} & \text{if } \E'\cong \cO(d_1-d+i) \oplus \cO(d_1-i) \text{ with $i=1, \ldots, \lfloor \frac{d-1}{2}\rfloor $.}  
        \end{cases}\]         
\end{prop}

\begin{proof}
   For  $\cE' \cong \cO(a)\oplus \cO(b) \in \Bun_2 (\P^1)$, there are four possibilities for the Hall product $\cK_x*\cE'$, as follows. 
   
If $a+d < b$,
\[ \begin{aligned}
    \cK_x*\cE'= & 
            (\cO(a+d)+q^d\cO(a)\oplus\cK_x)*\cO(b) \\
            =& \cO(a+d)\oplus \cO(b)+q^d\cO(a)\oplus \cO(b+d)+q^{2d}\cO(a)\oplus \cO(b)\oplus \cK_x. 
            \end{aligned}\]

If $a+d=b$,
\[ \begin{aligned}
  \cK_x*\cE'= (q+1)\cO(a+d)\oplus \cO(b)+q^d\cO(a)\oplus \cO(b+d) +q^{2d}\cO(a)\oplus \cO(b)\oplus \cK_x.
\end{aligned}\]

If $a+d>b$,
\[ \begin{aligned}
    \cK_x*\cE' &=\cO(a+d)*\cO(b)+ q^d\cO(a)*(\cO(b+d)+q^d\cO(b)\oplus \cK_x)\\
     & =  q^d\cO(b)\oplus \cO(a+d) +\sum_{i=1}^{\lfloor \frac{a+d-b}{2}\rfloor} (q^2-1)q^{a+d-b-1}\cO(b+i)\oplus \cO(a+d-i)\\
            & \quad +  q^d\cO(a)\oplus \cO(b+d)+q^{2d}\cO(a)\oplus \cO(b) \oplus  \cK_x. \\
\end{aligned}\]

If $a=b$,
\[ \begin{aligned}
    \cK_x*\cE' =& \tfrac{q-1}{q^2-1}\cK_x*\cO(a)* \cO(a) \\
             & = q^d\cO(a)\oplus\cO(a+d)+ \sum_{i=1}^{\lfloor d/2\rfloor}(q-1)q^{d-1}\cO(a+i)\oplus \cO(a+d-i) \\ 
             & \quad + q^{2d}\cO(a)\oplus\cO(a)\oplus \cK_x. \\
\end{aligned}\]

The proposition follows by examining the manner in which the vector bundle $\cO(d_1)\oplus \cO(d_2)$ appears in the above products. 
\end{proof}


\begin{rem}
    Note that in every case of the above propostion, in agreement with Theorem \ref{thm-grass}, given $\E \in \Bun_2\P^1$ 
    $$ \sum_{\E' \in \Bun_2(\P^1)} m_{x,1}(\E,\E') = q^d + 1 = \# \Gr (1,2) (\kappa(x)).$$
\end{rem}


\begin{thm}\label{hpro}
    Let $x \in \P^1$ be a closed point of degree $d$ and $\E := \bigoplus_{i=1}^n \cO (d_i) \in \Bun_n\P^1$. Write 
    \[\E \cong \bigoplus_{i=1}^{l_1} \cO (b_1) \oplus \bigoplus_{i=1}^{l_2} \cO (b_2) \oplus \cdots \oplus \bigoplus_{i=1}^{l_m} \cO (b_m),\] 
    with $b_i < b_j$ for $i < j$. Then 
\begin{equation}\label{prodform}
    \cK_x^{\oplus r} * \E= Q(\E)\sum_{i=0}^r \sum _{\sigma \in \Delta_i^n}q^{|\sigma|d} \hprod_{j=1}^n \cO(d_j+\sigma(j)d)*\cK_x^{r - i},
\end{equation} where $$ Q(\E)=\prod_{i=1}^m \prod_{j=0}^{l_i-1} \frac{q-1}{q^{l_i-j}-1}.$$
\end{thm}

\begin{proof}
Since $b_i < b_j$ for $i < j$, Theorem \ref{hallprop} items $(i)$ and $(iii)$, yields
\[\E =  \hprod_{i=1}^m\bigoplus_{j=1}^{l_i} \cO (b_j) = \hprod_{i=1}^{m} \left[ \left( \prod_{j=0}^{l_i-1}\dfrac{q-1}{q^{l_j-k}-1}\right) \hprod_{k=1}^{l_i}\cO(b_i) \right] = Q(\E)\hprod_{k=1}^n \cO(d_k).\]
Thus 
\begin{equation}\label{prodin}
     \cK_x^{\oplus r}*\E = Q(\E) \; \cK_x^{\oplus r}* \left( \hprod_{j=1}^n \cO(d_j)\right).
\end{equation} 

Let $\cL \in \Pic(\P^1)$ and $s \in \Z_{>0}$ with $s \leq r$. By item $(v)$ of Theorem \ref{hallprop}, 
\[\cK_x^{\oplus s}*\cL = \left( \cL \otimes \cO(d) \right) *\cK_x^{s-1}+q^{ds} \; \cL *\cK_x^{s}.\]
This means that the product of \( \mathcal{K}_x^{\oplus s} \) with a line subbundle of \( \mathcal{E} \) consists of two terms: (i) the direct sum of a line bundle whose degree is increased by \( d \), together with a torsion term whose weight is reduced by one, this term does not affect the Hall number, and; (ii) the direct sum of that line bundle with \( \mathcal{K}_x^{\oplus s} \), this term multiplies the Hall number by \( q^{ds} \). Since the reduction in the weight of skyscraper sheaf can occur at most \( r \) times, this process can only occurs at most \( r \) iterations.

Hence, we can choose to add \( d \) to the degrees of different invertible sheaves at most \( r \) times. Thus, the terms in the Hall product \eqref{prodin} are in bijection with elements of \( \bigcup_{i=0}^r \Delta_i^n \), associating the function \( \delta \in \Delta_s^n \) to  
\[
\left( \hprod_{j=1}^n \mathcal{O}(d_j + \delta(j)) \right) \oplus \mathcal{K}^{r - s}.
\]
Moreover, when the degree of an invertible sheaf \( \mathcal{L} \) does not change, the multiplicity remains the same. However, if the degree of \( \mathcal{L} \) increases by \( d \), the multiplicity is multiplied by \( q^{td} \), where \( t \) is equal to \( s \) minus the number of times the degrees of previous line subbundles of \( \mathcal{E} \) have changed. Therefore, the term associated to \( \delta \in \Delta_s^n \) will have the following multiplicity
\[ q^{(1-\delta(1))(s-\delta(1))d} \; q^{(1-\delta(2))(s-\delta(1)-\delta(2))d} \cdot \; \cdots \; \cdot  \ q^{(1-\delta(n))(s-\sum_{j=1}^n\delta(j))d} =q^{|\delta|d},\] 
and the weight of the skyscraper sheaf will have been reduced by \( s \), giving us the formula stated in the theorem.
 \end{proof}


\begin{cor}\label{prodsemtor}  In the notation of previous theorem, 
let $\pi^{vec}(\cK_x^{\oplus r}* \E)$ be the torsion-free terms in the product $\cK_x^{\oplus r}* \E$. Then
\[
    \pi^{vec}(\cK_x^{\oplus r}* \E) = Q(\E) \sum_{\delta \in \Delta_r^n}q^{|\delta|d} \hprod_{j=1}^n \cO(d_j+\delta(j)d)
\]
\end{cor}

\begin{proof} The proof follows by observing that the only terms in which the torsion sheaf $\cK_x$ does not appears in equation \eqref{prodform} are when $i=r$.    
\end{proof}


 \begin{df}
    Let $\E, \E' \in \Bun_n(\P^1)$, where $ \E' \cong \bigoplus_{i=1}^n \cO(d_i')$. Suppose that $[\E' \xrightarrow[r]{x} \E]$.  We say that $\delta \in \Delta_r^n$ \textbf{realizes} $[\E' \xrightarrow[r]{x} \E]$ if there exists $a\in \Z_{>0}$ such that $a\E$ appears in the Hall product
    \[\hprod_{i=1}^n \cO(d_i'+|x|\delta(i)). \]
    We denote the set of functions $\delta \in \Delta_r^n$ that realizes $[\E' \xrightarrow[r]{x} \E]$ by $\Delta_r^n(\E,\E',x)$. An element $\delta \in \Delta_r^n(\E', \E, x)$ is said to be $\textbf{maximal}$ if $|\delta| \geq |\sigma|$ for all $\sigma \in \Delta_r^n(\E', \E, x) $.
\end{df}


In order to prove the next theorem, we will need the following lemma.

\begin{lemma}
    As a set, the Grassmannian $\Gr(k,n)$ has a decomposition as the disjoint union
\[\Gr(k, n) = \bigsqcup_{\lambda \in J(k,n)} C_{\lambda}.\]
Where $J(k, n) = \{\lambda = (j_1, \ldots , j_k)\,\big|\, 1 \leq j_1 < \cdots < j_k \leq n\}$, and $C_{\lambda}$ denotes the set of $n \times n$-matrices $(a_{ij})_{n\times n}$ of the following form:
\begin{itemize}
    \item $a_{i,i}=1$ if $i \in \lambda$.
    \item $a_{i,j}=0$ if $j\in \lambda$, or $j<i$, or $i \in \lambda$ and $j\in \lambda$,
\end{itemize}
where we write $i \in \lambda$ to mean that $i$ appears as a entry of $\lambda$.
\end{lemma}

\begin{proof} This is the standard Schubert-cell decomposition of the Grassmannian, see for instance \cite[Lemma 2.2]{alvarenga19}.
\end{proof}


\begin{rem}\label{schubert} 
Note that each $\lambda=(j_1, \ldots, j_k) \in J(k,n)$ can be identified with a $\delta_{\lambda} \in \Delta_k^n$, defined by 
$$\delta_{\lambda}(i)=1 \quad \text{if and only if} \quad  i \in \lambda.$$

We can represent a matrix in $C_{\lambda}$ as
\[C_{\lambda} = \left\{\begin{pmatrix}
      \delta_{\lambda}(1) & b_{12} & \ & \ & \cdots & b_{1n-1} \\
      \ & \delta_{\lambda}(2) & b_{23} & \ & \cdots & b_{2n-1} \\
      \ & \ & \ddots & \ & \ & \vdots \\
      \ & \ & \ & \ & \cdots & \delta_{\lambda}(n) \\
    
        \end{pmatrix}  \ \Bigg|  \ \ \ \ \begin{matrix}
            b_{ij} \in \FF_q \ \text{ with } \ \\
            b_{ij}=0 \text{ if $j<i$ or } \\
            \delta_{\lambda} (i)= 0, \text{ or  } \\ \delta_{\lambda} (i)=1,
            \text{ and } \delta_{\lambda}(j)= 1
        \end{matrix} \right\}.\]

Define the function $\omega: \Delta_r^n \to \Z$ by
$$\omega(\sigma)= \sum_{\{i\,|\,\sigma(i)=1\}} i.$$

Let $\delta :=(0, \ldots, 0, 1,\ldots , 1)\in \Delta_r^n$. Note that $\omega(\delta) \geq \omega(\delta_{\lambda})$ for all $\lambda \in J(k,n)$. The number  \( \omega(\delta) - \omega(\delta_{\lambda}) \) counts how many entries equal to 1 in \( \delta \) are moved to the left in \( \delta_{\lambda} \). Therefore, the number of free entries for a matrix in $C_{\lambda}$ is 
$$q^{\omega(\delta)-\omega(\delta_{\lambda})}.$$ 
In particular, the previous lemma implies
\[\#\Gr(n,k)=\sum_{\sigma \in \Delta_r^n} q^{\omega(\delta)-\omega(\sigma)}.\]
\end{rem}


\begin{thm}\label{spacedmultip} Let $x \in \P^1$ be a closed point of degree $d$ and $\E := \bigoplus_{i=1}^n \cO (d_i) \in \Bun_n\P^1$. Suppose that there exists an index $n_1 \in \{2, \ldots , n-1 \}$ such that $d_{n_1+1} - d_{n_1} \geq d$. 
Let $\E' := \bigoplus_{i=1}^n \cO(d_i') \in \Bun_n (\P^1)$ and define 
$$\E_1 := \bigoplus_{i=1}^{n_1} \cO(d_i),\ \ \E_1' := \bigoplus_{i=1}^{n_1} \cO(d_i'), \ \  
\E_2 := \bigoplus_{i=n_1+1}^n \cO(d_i) \ \ \text{and} \ \ \E_2' := \bigoplus_{i=n_1+1}^n \cO(d_i').$$
 Then   
 $$m_{x,r}(\E', \E)=m_{x,r_1}(\E'_1, \E_1) \ m_{x,r_2}(\E_2',\E_2) \ q^{r_2(n_1-r_1)|x|}.$$ 
where   $r_1 := (\deg(\E_1)- \deg(\E_1'))/d$ and $r_2:=r-r_1$. 
    
\end{thm}

\begin{proof}
    The Corollary \ref{break} implies that $m_{x,r} (\E', \E) \neq 0 $ if and only if $m_{x,r_1} (\E_1', \E_1) \neq 0 $ and $m_{x,r_2} (\E_2', \E_2) \neq 0 $. Then, the equality is well-defined even if $\E'$ is not a Hecke modification of $\E$.
    
    First, suppose that $d_{n_1+1}-d_{n_1} > d$. By Corollary \ref{greater0}, the condition $d_{n_1+1}-d_{n_1} >|x|$ implies that $d_i'<d_{\ell}'$ for all $i \in \{1, \ldots, n_1\}$ and $\ell \in \{n_1+1, \ldots, n\}$. Our goal is to apply the formula of Corollary \ref{prodsemtor} to compute $m_{x,r}(\E',\E)$ in function of $m_{x,r_1}(\E_1',\E_1)$ and $m_{x,r_2}(\E_2',\E_2)$. Denote:
    \begin{align*}
        c(\E'):= & \sum_{\sigma \in\Delta_r^n(\E',\E)}q^{|\sigma|}\hprod_{j=1}^n \cO(d_j+\sigma(j)d), \\
        c(\E_1'):= & \sum_{\sigma_1 \in\Delta_{r_1}^{n_1}(\E_1',\E_1)}q^{|\sigma_1|}\hprod_{j=1}^{n_1} \cO(d_j+\sigma_1(j)d), \\
        c(\E_2'):=& \sum_{\sigma_2 \in\Delta_{r_2}^{n-n_1}(\E_2',\E_2)}q^{|\sigma_2|}\hprod_{j=1}^{n-n_1} \cO(d_{n_1+j}+\sigma_2(n_1+j)d).
    \end{align*}

    By the properties of the Hall product in Theorem \ref{hallprop},
    \begin{equation}\label{breakbreak}
        Q(\E') = Q(\E_1') Q(\E_2')
    \end{equation} 
    and
    \begin{equation}\label{prodbreak}
   \sum _{\sigma \in \Delta_{r}^n} \hprod_{i=1}^n \cO(d_i)=  \sum _{\sigma \in \Delta_{r}^n} \left( \hprod_{i=1}^{n_1} \cO(d_i'+\sigma(i)d) \bigoplus \hprod_{i=n_1+1}^n \cO(d_i'+\sigma(i)d)\right). 
    \end{equation}
    
Next we observe that any $\delta \in \Delta_r^n(\E',\E,x)$ can be writen as a concatenation of an element $\delta_1 \in \Delta_{r_1}^{n_1}(\E_1',\E_1,x)$ and an element $\delta_2 \in \Delta_{r_2}^{n-n_1}(\E_2',\E_2,x)$. Thus, 
\begin{align*}
    |\delta|=\sum_{i=1}^n[(1-\delta(i))(r-\sum_{\ell=1}^i \delta(\ell))], \ \ \ \  & \\
    |\delta_1|=\sum_{i=1}^{n_1}[(1-\delta_1(i))(r_1-\sum_{\ell=1}^i \delta(\ell))] & =\sum_{i=1}^{n_1}[(1-\delta(i))(r_1-\sum_{\ell=1}^i \delta(\ell))], \\
    |\delta_2|=\sum_{i=1}^{n_2}[(1-\delta_2(i))(r_2-\sum_{\ell=1}^i \delta(\ell))] & =\sum_{i=n_1+1}^{n}[(1-\delta(i))\underbrace{(r_2-\sum_{\ell=n_1+1}^i \delta(\ell))]}_{r-\sum_{\ell=1}^n \delta(\ell)}.
\end{align*}
Furthermore,
\[|\delta| - |\delta_1|- |\delta_2|= \sum_{i=1}^{n_1}(1-\delta(i))(r-r_1)=r_2(n_1-r_1).\]

Hence, for any $\delta$ that realizes the Hecke modification $[\E' \xrightarrow[r]{x} \E]$ and appears in the formula of Corollary \ref{prodsemtor}, the multiplicity of the term corresponding to $\E$ satisfies
\[ q^{|\delta|d}\hprod_{i=1}^n \cO(d_i'+\delta(i)d) = 
    q^{(|\delta_1|+|\delta_2|+r_2(n_1-r_1))d}\hprod_{i=1}^{n_1} \cO(d_i'+\delta_1(i)d) * \hprod_{i=n+1}^n \cO(d_i'+\delta(i)d).\]

Note that the exponent $r_2(n_1-r_1)$ does not depends on $\delta$, then 
\begin{align*}
    Q(\E') \ c(\E') = \ q^{(r_2(n_1-r_1))d} \  
    Q(\E_1') \ c(\E_1') \ Q(\E_2') \ c(\E_2'),
\end{align*}
which completes the proof in this case. 

Now suppose that $d_{n_1+1} - d_{n_1} = d$. If $d_{n_1+1}'> d_{n_1}'$, the identities \eqref{breakbreak} and \eqref{prodbreak} are true, then we can apply the same argument as above. Therefore, we are left to the case when $d_{n_1+1}'= d_{n_1}'$, that is $d_{n_1}'=d_{n_1}=d_{n_1+1}'$, and $d_{n_1+1}=d_{n_1}+d$. 

Let $\alpha$ be the number of line bundles $\cO(d_i')$  with $i \in \{1, \dots n_1\}$ such that $d_i'=d_{n_1}$. Observe that $\alpha \leq n_1-r_1$. Similarly, let $\beta$ be the number of line bundles $\cO(d_i')$ with $i \in \{n_1+1, \ldots, n\}$ such that $d_i'= d_{n_1}$. Observe that $\beta \leq r_2$. Since the greater degree in $\E_1$ is $d_{n_1}$, if $\sigma_1 \in \Delta_{r_1}^{n_1}(\E_1',\E_1)$, then
\[\sigma_1=(\ell_1, \ldots, \ell_{n_1-\alpha}, \underbrace{0, 0, \ldots, 0}_{\alpha \ \text{times}}).\]
Since the smaller degree in $\E_2$ is $d_{n_1}+d$, if $\sigma_2 \in \Delta_{r_2}^{n_2}(\E_2',\E_2)$, then
\[\sigma_2=(\underbrace{1, 1, \ldots, 1}_{\beta \ \text{times}}, \ell_{n_1+\beta+1}, \ldots, \ell_{n}).\]
 Where $\ell_i \in \{0,1\}$ for $i=1, \ldots, n$. 
 
 We denote the concatenation of $\sigma_1 \in \Delta_{r_1}^{n_1}(\E_1',\E_1)$ and $\sigma_2 \in \Delta_{r_2}^{n_2}(\E_2',\E_2)$ by $\sigma_1 \oplus \sigma_2$. Let $\sigma = \sigma_1 \oplus \sigma_2 \in \Delta_r^n$. Then $\sigma$ realizes the Hecke modification $[\E' \xrightarrow[r]{x} \E]$. We observe that unlike the previous case, not all elements in $\Delta_{r}^n(\E',\E)$ can be obtained by concatenation of such elements $\sigma_1$ and $\sigma_2$.
Let $m :=\min\{\alpha, \beta\}$ and $t\in \{1, \ldots, m \}$. We can construct a $\sigma' \in \Delta_r^n(\E',\E)$ by exchanging $t$ elements equals to $1$ in $\sigma_2(1), \ldots, \sigma_2(\beta)$ with $t$ elements equals to $0$ in $\sigma_1(n_1-\alpha), \ldots, \sigma_1(n_1)$. Note that this construction gives us $\sigma_1' \in \Delta_{r_1+t}^{n_1}$ and $\sigma_2'\in \Delta_{r_2-t}^{n-n_1}$ such that $\sigma'=\sigma_1'\oplus \sigma_2'$. For each choice of changing positions, we obtain an element $\delta \in \Delta_{\beta}^{\alpha+\beta}$, such that
\[\sigma'=(\ell_1, \ldots, \ell_{n_1-\alpha})\oplus \delta \oplus (\ell_{n_1+\beta+1}, \ldots, \ell_n).\] 
We denote the corresponding element in $\Delta_{r}^n(\E',\E)$ by $\sigma_{\delta}$, with $\sigma$ associated with
\[\delta_0=(\underbrace{0, \ldots, 0}_{\alpha \ \text{times}}, \underbrace{1, \ldots, 1}_{\beta \ \text{times}}).\]
Let $\omega$ be the function defined in Remark \ref{schubert} and let $s(\delta):=\omega(\delta_0)-\omega(\delta)$. Then 
\[|\sigma_{\delta}|=|\sigma|-s(\delta)\]
and
\[\hprod_{i=1}^n\cO(d_i'+\sigma(i)d)=q^{s(\delta)(d+1)}\hprod_{i=1}^n\cO(d_i'+\sigma_{\delta}(i)d).\]
Since any $\rho \in \Delta_{r}^n(\E',\E)$ is of the form $\sigma_{\delta}$, for some $\delta \in \Delta_{\alpha}^{\alpha+\beta}$ and $\sigma=\sigma_1 \oplus \sigma_2$, with $\sigma_1 \in \Delta_{r_1}^{n_1}(\E_1',\E_1)$ and $\sigma_2 \in \Delta_{r_2}^{n-n_1}(\E_2',\E_2)$, then
\begin{align*}
   c(\E')= & c(\E_1') *  c(\E_2')\;  q^{r_2(n_1-r_1)d} \sum_{\delta \in \Delta_{\beta}^{\beta+\alpha}} q^{-s(\delta) d} q^{s(\delta)(d+1)}  \\
   =& \sum_{\sigma_1 \in \Delta_{r_1}^{n_1}(\E_1',\E_1)} \ \ \sum_{\sigma_2 \in \Delta_{r_2}^{n_2}(\E_2',\E_2)}\left( q^{(|\sigma_1|+|\sigma_2|+r_2(n_1-r_1)) d} \hprod_{i=1}^n\cO(d_i'+\sigma_1(i)d)\right) \sum_{\delta \in \Delta_{\beta}^{\beta+\alpha}} q^{-s(\delta)}.
\end{align*}
Remark \ref{schubert} implies that $\sum_{\delta \in \Delta_{\beta}^{\beta+\alpha}} q^{-s(\delta)}=\# \Gr(\beta,\alpha+\beta)$.

If we represent $\E' \cong \bigoplus_{i=1}^{l_1} \cO (b_1) \oplus \bigoplus_{i=1}^{l_2} \cO (b_2) \oplus \cdots \oplus \bigoplus_{i=1}^{l_m} \cO (b_m)$, as in Theorem \ref{hpro}, where $b_i < b_k$ for $i < k$. Then,
$$ Q(\E')=\prod_{i=1}^m \prod_{k=0}^{l_i-1} \frac{q-1}{q^{l_i-k}-1}.$$
We can also decompose $\E_1'$ and $\E_2'$ as a sum of line bundles of same degree
\[\E_1' \cong \bigoplus_{i=1}^{l_1} \cO (b_1) \oplus \cdots \oplus \bigoplus_{i=1}^{l_{m_1}} \cO (b_{m_1}), \ \ \E_2' \cong \bigoplus_{i=1}^{l_{m_1+1}} \cO (b_{m_1+1})  \oplus \cdots \oplus \bigoplus_{i=1}^{l_m} \cO (b_m).\]
With $b_{m_1}=b_{m_{1}+1}=d_{n_1}$, $l_{m_1}=\alpha$ and $l_{m_1+1}=\beta$. This implies that

\begin{align*}
    \dfrac{Q(\E')}{Q(\E_1')Q(\E_2')} & = \prod_{k=0}^{\alpha+\beta-1}\dfrac{q-1}{q^{\alpha+\beta - k}-1} \ \ \prod_{k=0}^{\alpha-1}\dfrac{q^{\alpha- k}-1}{q-1}  \ \  \prod_{k=0}^{\beta-1}\dfrac{q^{\beta - k}-1}{q-1}&\\
     & = \prod_{k=0}^{\beta-1}\dfrac{q^{\beta-k}-1}{q^{\alpha+\beta - k}-1}& \\
     & = \# \Gr(\beta, \alpha+\beta)^{-1}.&
\end{align*}
Therefore 
\begin{align*}
    Q(\E') c(\E')= \dfrac{q^{(r_2(n_1-r_1))d}\# \Gr(\beta, \alpha+\beta) 
    Q(\E_1') c(\E_1') Q(\E_2') c(\E_2')}{\#\Gr(\beta, \alpha+\beta)},
\end{align*}
which implies 
\[ m_{x,r}(\E', \E)=m_{x,r_1}(\E'_1, \E_1) \; m_{x,r_2}(\E_2',\E_2) \; q^{r_2(n_1-r_1)d},\] 
as desired.
\end{proof}


An application of computing such Hall numbers is the following proposition, which was inspired in an  example from \cite{baumann-kassel-01}. Denote by $\FF_q[S,T]_d^h$ the set of homogeneous polynomials of degree $d$ in the variables $S,T$ over $\FF_q$. Moreover, we denote by $\mathrm{Mat}_n(\FF_q[S,T])$ the set of $n \times n$ matrices over $\FF_q[S,T]$. 

\begin{prop} Let $x \in \P^1$ be a closed point of degree $d \geq 2$. 
Let $F(S,T) \in \FF_q[S,T]_d^h$ irreducible corresponding to $x$ and $\E' = \bigoplus_{i=1}^n \cO(a_i), \E = \bigoplus_{i=1}^n \cO(b_i) \in \Bun_n (\P^1)$. Then the number of monomorphisms $\varphi: \E' \to \E$  such that $\det(\varphi)=u F(S,T)$, where $u \in \FF_q^*$, is 
\[ m_{x,1}(\E',\E) \cdot \#\Aut(\E').\] 
To be more precise, with respect to the left group action of $\Aut(\E')$ on $\rm{Mat}_n(\FF_q[S,T])$, there are
exactly $m_{x,1}(\E,\E')$ classes of matrices
\[\overline{\varphi} \in \mathrm{Mat}_n(\FF_q[S,T]) / \Aut(\E') \ \text{ such that } \ \det(\varphi)=u F(S,T).\]
Moreover, up to column permutations,
\[\varphi_{ij} \in \left\{
\begin{array}{ll}
\FF_q[S,T]_{b_j-a_i}^h & \text{if} \ a_i-b_j \geq 0, \\
\{0\} & \text{if} \ a_i-b_j < 0.
\end{array} \right.\]
where $\varphi = (\varphi_{ij})_{ij}$. 
\end{prop}

\begin{proof} Since    
\[
    \Hom(\cO(a_i),\cO(b_j))=\left\{
\begin{array}{ll}
\FF_q[S,T]_{b_j-a_i}^h & \text{if} \ a_i-b_j \geq 0 \\
\{0\} & \text{if} \ i-j < 0
\end{array} \right. .
\] 
The proposition follows  from the Smith Normal Form of the morphism $ \varphi: \E' \to \E $ and from the fact that $\varphi$ induces a Hecke modification $[\E' \xrightarrow[]{x} \E]$. 
\end{proof}


\section{Explicit Hecke modifications}
\label{sec-explicit calculations}

In this section, we continue to assume that \( k = \mathbb{F}_q \) is the finite field with \( q \) elements, and we let \( X \) be the projective line over \( \mathbb{F}_q \).
We apply the structure results from previous section to describe (explicitly) the Hecke modifications, and their multiplicities, for every vector bundle over $\P^1$.

\begin{thm} Let $x \in \P^1$ be a closed point of degree $d$, and let $\E\in \Bun_n\P^1$ be such that $\E= \bigoplus_{i=1}^n \cO (d_i)$. Suppose $d_{i+1}-d_i \geq d$ for all $i \in \{1,2, \ldots, n-1\} $. Given $\delta \in \Delta_r^n$, define 
    $$\cE' :=  \bigoplus_{i=1}^n \cO(d_i-\delta(i)d).$$ 
    Then the multiplicity of the Hecke modification $[\E' \xrightarrow[r]{x} \E]$ is given by
    \[m_{x,r}(\E', \E)=q^{|\delta|d}.\]
    
\end{thm}

\begin{proof} First, let 
\[ \E_j' := \bigoplus_{i=j}^n \cO(d_i-\delta(i)|x|) \quad \text{and}  \quad \E_j :=\bigoplus_{i=j}^n \cO(d_i).\]
Lemma \ref{existtrivial} and Theorem \ref{break} yields 
$ m_{x,r_j}(\E_j',\E_j) \neq 0 $
for all $j \in \{1, \ldots, n\}$, where $r_j=r-\sum_{i=1}^j \delta(i)$. 
Furthermore,
\[m_{x, \delta(i)}(\cO(d_i-\delta(i)d),\cO(d_i))=1 \ \ \text{for all $i \in \{1, \ldots, n\}$}.\]
Let $c(j) := q^{r_j(1-\delta(j))d}.$ By Theorem \ref{spacedmultip},
\begin{align*}
    m_{x,r}(\E',\E) & =  m_{x, \delta(1)}(\cO(d_1-\delta(1)d),\cO(d_1)) \ m_{x,r-\delta(1)}(\E_1',\E_1) \ c(1)\\
   & = m_{x,\delta(2)}(\cO(d_2-\delta(2)d),\cO(d_2)) \ m_{x,r-\delta(1)-\delta(2)}(\E_2',\E_2) \ c(1) \ c(2) \\
   & \; \; \vdots \\ 
   & =  m_{x,\delta(j)}(\cO(d_j-\delta(j)d), \cO(d_j)) \ m_{x, r-\sum_{i=1}^j \delta(i)}(\E_j',\E_j) \ \prod_{i=1}^j c(i).
\end{align*}
Therefore,
$$m_{x,r}(\E',\E)=\prod_{i=1}^n c(i)=q^{(\sum_{i=1}^n(r-\sum_{\ell=1}^i\delta(\ell))(1-\delta(i)))d}=q^{|\delta|d},$$
which establishes the formula.
\end{proof}

 \begin{thm}
    Let $x \in \P^1$ be a closed point of degree one, and let $\E \in \Bun_n (\P^1)$ be such that $\E = \bigoplus_{i=1}^{n} \cO(d_i)$. Given a Hecke modification $[\E' \xrightarrow[]{x} \E]$, there exists a Kronecker delta function $\delta_{j} \in \Delta_1^n$ such that 
    \[ \E' \cong \ \bigoplus_{i=1}^{n} \cO(d_i-\delta_{j}(i)).\] 
     Furthermore, let $A :=\{k \in \{1, 2, \ldots, n\}\,\big|\,  d_k=d_j\}$. Then
     \[m_{x,1}(\E, \E')=q^{\min(A)-1}\;\dfrac{q^{\#A}-1}{q-1}.\]
 \end{thm}

\begin{proof}
 According to Corollary $\ref{greater0}$, if $[\E' \xrightarrow[]{x} \E]$, then 
 \[ \E' \cong \bigoplus_{i=1}^{n} \cO(d_i-\delta_{j}(i)),\] 
 for some $\delta_{j} \in \Delta_1^n$. Observe that $\delta_j$ denotes the Kronecker delta given by $\delta_j(i)=1$ if $i=j$ and $\delta_j(i) = 0$ if $i \neq j.$
 
Write,
\[ \E = \bigoplus_{i=1}^n \cO (d_i) = \bigoplus_{i=1}^{l_1} \cO (b_1) \oplus \bigoplus_{i=1}^{l_2} \cO (b_2) \oplus \cdots \oplus \bigoplus_{i=1}^{l_m} \cO (b_m),\] with $b_i < b_k$ for $i<k$. Let $s \in \{1, \ldots, m\}$ such that $b_s=d_j$. Hence, $l_s=\#A$. We can decompose $\E = \E_1 \oplus \E_s \oplus \E_2$, where
 \[\E_1 = \bigoplus_{k=1}^{s-1} \bigoplus_{i=1}^{l_{k}} \cO (b_{k}), \quad  
 \E_s= \bigoplus_{i=1}^{l_s} \cO (d_j), \quad 
 \E_2 = \bigoplus_{k=s+1}^{m} \bigoplus_{i=1}^{l_{k}} \cO (b_{k}). \]
Let $\E_1' :=\E_1$, $\E_2' :=\E_2$ and $\E_s' : =\bigoplus_{i=1}^{s} \cO (d_j-\delta_j(\rk(\E_1)+i))$. Thus, $ \E'=\E_1' \oplus \E_s' \oplus \E_2'$. 
Theorem \ref{spacedmultip} yields
 \begin{align*}
     m_{x,1}(\E',\E)=& m_{x,0}(\E_1',\E_1) \ m_{x,1}(\E_s'\oplus \E_2',\E_s \oplus \E_2) \ q^{\rk(\E_1)}\\
     =& m_{x,0}(\E_1',\E_1) \ m_{x,1}(\E_s',\E_s) \ m_{x,0}(\E_2',\E_2) \ q^{\rk(\E_1)} \\
     =& q^{\rk(\E_1)} \ m_{x,1}(\E_s',\E_s),
 \end{align*} 
where the last equality follows from the fact that 
\[ m_{x,0}(\E_1',\E_1) = \ m_{x,0}(\E_2',\E_2)= 1. \]

Next, we observe that $\rk(\E_1) = \text{max}\{i \in \{1, \ldots, n\}\,\big|\, d_i<d_j\}=\min(A)-1$. Moreover, 
\[ \E_s' \cong \cO(d_j-1)\oplus \cO(d_j) \oplus \cO(d_j) \oplus \cdots \oplus \cO(d_j), \] thus 
\[ Q(\E_s')= \prod_{i=0}^{l_s-2}\frac{q-1}{q^{l_s-1-i}-1},\]
where $Q(\E_s')$ is as in Theorem \ref{hpro}
Thus, in order to compute the multiplicity $m_{x,1}(\E_s',\E_s)$, following Corollary \ref{prodsemtor}, we only need to consider $\delta=(1,0, \ldots, 0) \in \Delta_1^{l_s}$. Therefore, 
\[m_{x,1}(\E_s',\E_s)=Q(\E_s') \hprod_{i=1}^{l_s}\cO(d_j)q^{|\delta|} = \prod_{i=0}^{l_s-2}\frac{q-1}{q^{l_s-1-i}-1} \prod_{i=0}^{l_s-1}\frac{q^{l_s-i}-1}{q-1}=\frac{q^{l_s}-1}{q-1}.\]
Since $l_s=\#A$, we obtain the desired result.\end{proof}


\begin{lemma}\label{multgrass} 
   Let  $x \in \P^1$ be a closed point of degree $d$, and let $\E, \E'\in \Bun_n(\P^1)$ be such that
$\E := \bigoplus_{i=1}^n\cO$ and  $ \E' := \bigoplus_{i=1}^r\cO(-x) \oplus \bigoplus_{i=1}^{n-r} \cO.$
   Then
    \[m_{x,r}(\E', \E)= \#\Gr(r,n)(\FF_q).\]
\end{lemma}
\begin{proof}
    Theorem \ref{hpro} yields
    \[Q(\E)=\prod_{i=0}^{n-1} \dfrac{q-1}{q^{n-i}-1} \ \ \text{and} \ \ Q(\E')=\prod_{i=0}^{r-1} \dfrac{q-1}{q^{r-i}-1}\prod_{i=0}^{n-r-1} \dfrac{q-1}{q^{n-r-i}-1}.\]
    The only function $\delta \in \Delta_r^n$ that realizes the Hecke modification $[\E'\xrightarrow[r]{x} \E]$ is
    \[\delta=(\underbrace{1, \ldots, 1}_{r \  \text{times}}, \underbrace{0, \ldots,0}_{n-r \ \text{times}}), \]
and  $|\delta|=0.$    
    Then
    \begin{align*}
        m_{x,r}(\E',\E) & = q^{d |\delta|} Q(\E')Q(\E)^{-1}=\prod_{i=0}^{n-r-1}\dfrac{q^{n-i}-1}{q^{r-i}-1}
         =\#\Gr(n-r,n)(\FF_q ).
    \end{align*}
The lemma follows from the fact that $\#\Gr(n-r,n)(\FF_q ) =\#\Gr(r,n)(\FF_q)$.
\end{proof}


\begin{rem}
Note that while the sum of all the multiplicities of Hecke modifications of \( \mathcal{E} \) at the closed point \( x \) with weight \( r \) is \( \#\Gr(r,n)(\kappa(x)) \), the previous lemma computed the multiplicity of a specific Hecke modification, which is equal to \( \#\Gr(r,n)(\mathbb{F}_q) \). These two numbers are equal if and only if \( |x| = 1 \).
\end{rem}

 The following theorem gives an explicit classification - including multiplicities - of all Hecke modifications of rank \( n \) vector bundles over \( \mathbb{P}^1 \) at a closed point of degree one and for any weight \( r \).


\begin{thm}\label{deg1anyr} 
 Let  $x \in \P^1$ be a closed point of degree one, and let $\E \in \Bun_n (\P^1)$ be such that
$ \E = \bigoplus_{i=1}^{n} \cO(d_i)=\bigoplus_{i=1}^m \bigoplus_{j=1}^{\ell_i} \cO(b_i),$ where $i<j$ implies $b_i < b_j$. 
Then $\E'\in \Bun_n\P^1$ is a Hecke modification of $\E$ at $x$ of weight $r$ if and only if there exists $\delta \in \Delta_r^n$ such that 
\[ \E' \cong \bigoplus_{i=1}^n \cO(d_i-\delta(i)).\] Furthermore,
\[m_{x,r}(\E', \E)=q^{\alpha}\prod_{j=1}^m \#\Gr(\theta_j,\ell_j),\]
where $\theta_j=\#\{ i \in \{1, 2, \ldots, n\}\,\big|\, d_i=b_j \text{ and } \delta(i)=1\}$ and $ \alpha=\sum_{j=1}^m(\ell_j-\theta_j)(r-\sum_{i=1}^j \theta_i)$.
\end{thm}

\begin{proof} By Corollary $(\ref{greater0})$, $\E'= \bigoplus_{i=1}^n\cO(d_i-\delta(i))$ for some $\delta \in \Delta_r^n$. 
    
For each $j \in \{1, \ldots, m-1\}$, we define the vector bundles
 \begin{align*}
\F_j     &= \bigoplus_{i=1}^{\ell_j} \cO(b_j)  & \text{ and } & \quad 
&\E_j     &= \bigoplus_{k=j+1}^m \bigoplus_{i=1}^{\ell_k} \cO(b_k), \\
\F_j'    &= \bigoplus_{i=\ell_1 + \cdots + \ell_{j-1} + 1}^{\ell_1 + \cdots + \ell_j} \cO(d_i')  & \text{ and } & \quad 
&\E_j'    &= \bigoplus_{i=\ell_1 + \cdots + \ell_j + 1}^{n} \cO(d_i'), \\
\end{align*}
and the quantities
\[ r_j     = r - \sum_{i=1}^{j} \theta_i \quad \text{ and } \quad 
c(j)     = q^{r_j(\ell_j - \theta_j)}. \]
    
According to Lemma \ref{multgrass}, $m_{x, \theta_i}(\cF_i', \cF_i) = \# \Gr(\theta_i, \ell_i)$, for all $i \in \{1, \ldots, m\}$. Furthermore, Theorem \ref{spacedmultip} implies that
    \begin{align*}
        m_{x,r}(\E',\E) &=  m_{x, \theta_1}(\F_1',\F_1) \ m_{x,r_1}(\E_1',\E_1) \ c(1)\\
        &= \#\Gr(\theta_1,\ell_1) \ m_{x, \theta_2}(\F_2',\F_2) \ m_{x,r_2}(\E_2',\E_2) \ c(1) \ c(2) \\
        & \; \; \vdots \\
        &= \prod_{i=1}^{j-1}( c(i) \ \#\Gr(\theta_i, \ell_i)) \ m_{x, \theta_j}(\F_j',\F_j) \ m_{x,r_j}(\E_j',\E_j) \ c(j).\\
    \end{align*}
Hence,    
    \[m_{x,r}(\E',\E)= \prod_{i=1}^m \#\Gr(\theta_i, \ell_i) \; c(i)\]
and, finally, 
    \[\prod_{j=0}^m c(j)=q^{\sum_{j=1}^m(\ell_j-\theta_j)(r-\sum_{i=1}^j \theta_i)} ,\]
which completes the proof.
\end{proof} 


\begin{thm} \label{thm-C} Let  $x \in \P^1$ be a closed point of degree $d$, and let $\E', \E \in \Bun_n\P^1$ be such that $\E := \bigoplus_{i=1}^n \cO(d_i) $ and $ \E' := \bigoplus_{i=1}^n \cO(d_i')$. Let us suppose that $d_i'=d_i-\epsilon_i$, where $ \epsilon_i \in \{0, 1, \ldots, d\}$ for all $i \in \{1, \ldots, n\}$ and $\sum_{i=1}^n \epsilon_i= d$. Let $A: = \{i \in \{1, \ldots, n\}\,\big|\, \epsilon_i \neq 0\}$, $s : =\min(A)$ and $B:=\max(A)$. Then $\E'$ is a Hecke modification of $\E$ at $x$ with weight $1$ if and only if 
    \[d_{j+1}-\epsilon_{j+1} \leq d_j \ \ \text{for all $j \in \{s, s+1, \ldots , B-1\}$}.\]
\end{thm}

\begin{proof}
First, we note that, if $d_{s+i+1}-\epsilon_{s+i+1} \leq d_{s+i}$, then 
\[d_{s+i}+d -\sum_{j=0}^{i} \epsilon_{s+j} \geq d_{s+i+1}-\epsilon_{s+i+1} \] 
for all $i \in \{0, 1, \ldots, B-s-1\}.$

Let $\E$ and $\E'$ as in the statement. By Corollary $\ref{prodsemtor}$, in order to $\E'$ be a Hecke modification of $\E$ at $x$ with weight one, we need to check the existence of $\Tilde{\sigma} \in \Delta_1^n$ such that $\E$ appears in the Hall product 
\[ \hprod_{i=1}^n\cO(d_i'+\Tilde{\sigma} d)\] 
with non-zero coefficient, i.e., that $\Tilde{\sigma} $ realizes $[\E' \xrightarrow[]{x} \E] $.
Let
\[\Tilde{\sigma}(i) := \left\{ \begin{array}{ll}
1, \ \ \ \  \text{if} \  i=s \\ 0,  \ \ \  \ \text{otherwise.}   \end{array}\right.\]
Since the degrees \( d_i' \) are in increasing order, then there is an integer $a>0$ such that
\begin{equation}\label{prodpeso1}
    \hprod_{i=1}^n\cO(d_i'+\Tilde{\sigma} d)=a\bigoplus_{i=1}^{s-1}\cO(d_i')*\cO(d_s-\epsilon_s +d) * \hprod_{i=s+1}^{n}(d_i').
\end{equation}
Next, since $d_s+d -\epsilon_s \geq d_{s+1}'=d_{s+1}-\epsilon_{s+1}$, there exists an integer $b_{\ell}$ such that
\[\cO(d_s+d-\epsilon_s)*\cO(d_{s+1}-\epsilon_{s+1})=\sum_{\ell=0}^{r}b_{\ell} \; \cO(d_{s+1}-\epsilon_{s+1}+\ell) \oplus \cO(d_s+d-\epsilon_s -\ell),\] 
where $r=\left\lfloor\tfrac{ (d_s+d-\epsilon_s) +(d_{s+1}-\epsilon_{s+1})}{2}\right\rfloor $. 

Note that $d_{s+1}-\epsilon_{s+1} \leq d_s$. This implies that we can recover $d_s$ in the above summand. Thus,  there exists a unique $\ell_s \in \{0, 1, \ldots, r\}$ such that either
\[ d_s=d_s+d -\epsilon_s-\ell_s \quad \text{or} \quad  d_s=d_{s+1}-\epsilon_{s+1}+\ell_s.\]
In both cases, in the product $\cO(d_s+d-\epsilon_s)*\cO(d_{s+1}-\epsilon_{s+1})$, there exists a term 
\[b_{\ell_s} \; \cO(d_s)\oplus\cO(d_{s+1}+d-\epsilon_s-\epsilon_{s+1}),\]
for some nonzero $b_{\ell_s} \in \Z$. 
Similarly, we obtain 
\[d_{s+1}+d-\epsilon_s-\epsilon_{s+1} \geq d_{s+2}-\epsilon_{s+2} \quad \text{and} \quad d_{s+2}-\epsilon_{s+2} \leq d_{s+1}.\]
Thus, there exists an integer 
\[ \ell_{s+1}\in \left\{0, 1, \ldots , \left\lfloor\tfrac{ (d_{s+1}+d-\epsilon_s-\epsilon_{s+1}) +(d_{s+2}-\epsilon_{s+2})}{2} \right\rfloor \right\}\] 
such that in the Hall product  $\cO(d_{s+1}+d-\epsilon_s-\epsilon_{s+1})*\cO(d_{s+2}-\epsilon_{s+2})$ 
there is a term given by
\[ b_{\ell_{s+1}} \cO(d_{s+1})\oplus \cO(d_{s+2}+d-\epsilon_s-\epsilon_{s+1}-\epsilon_{s+2}),\]
for some nonzero $b_{\ell_{s+1}} \in \Z$.

Proceeding analogously, we obtain that
\[  d_{B-1}+d -\sum_{i=0}^{B-s-1}\epsilon_{s+i} \geq d_B-\epsilon_B \quad \text{and} \quad  d_B-\epsilon_B \leq d_{B-1}.\] 
Then, there exists an integer 
\[ \ell_{B-1} \in \left\{0, 1, \ldots, \left\lfloor \tfrac{(d_{B-1}+d -\sum_{i=0}^{B-s-1}\epsilon_{s+i}) - (d_B-\epsilon_B)}{2}  \right\rfloor \right\}\]
such that in the Hall product 
\[ \cO \big(d_{B-1}+d-\sum_{i=0}^{B-s-1} \epsilon_{s+i} \big) * \cO(d_{B}-\epsilon_{B})\] 
there is a term given by
\[ b_{\ell_{B-1}} \; \cO(d_{B-1})\oplus \cO(d_B+d-\sum_{i=0}^{B-s}\epsilon_{s+i}),\]
for some nonzero $b_{\ell_{B-1}} \in \Z$.
The definition of $s$ and $B$, yields $ \sum_{i=0}^{B-s}\epsilon_{s+i}=\sum_{i=1}^n \epsilon_{i}=d$. Then, 
\[  b_{\ell_{B-1}} \; \cO(d_{B-1})\oplus \cO(d_B+d-\sum_{i=0}^{B-s}\epsilon_{s+i}) 
=  b_{\ell_{B-1}} \; \cO(d_{B-1})\oplus \cO(d_B). \]
Hence, the Hall product \eqref{prodpeso1} has a term given by
\[ a\prod_{i=s}^{B-1}b_{\ell_i}\bigoplus_{i=1}^{s-1}\cO(d_i')*\bigoplus_{i=s}^B\cO(d_i) * \hprod_{i=B+1}^{n}\cO(d_i').\]

Lastly, since $d_i'=d_i$ for all $i \in \{1, \ldots, s-1\} \cup\{B+1, \ldots, n\}$ and the degrees $d_i$ are in increasing order, there exists a positive integer $c \in \Z$ such that 
\[ \hprod_{i=B+1}^{n}\cO(d_i') =c \bigoplus_{i=B+1}^{n}\cO(d_i).\]
Therefore, there exists an integer $b \in \Z$ such that
\begin{align*}
    a\prod_{i=s}^{B-1}b_{\ell_i}\bigoplus_{i=1}^{s-1}\cO(d_i')*\bigoplus_{i=s}^B\cO(d_i) *\hprod_{i=B+1}^{n}\cO(d_i') &=
ac\prod_{i=s}^{B-1}b_{\ell_i}\bigoplus_{i=1}^{s-1}\cO(d_i)*\bigoplus_{i=s}^B\cO(d_i) * \bigoplus_{i=B+1}^{n}\cO(d_i) \\
& = abc\prod_{i=s}^{B-1}b_{\ell_i}\bigoplus_{i=1}^{n}\cO(d_i) \\
& = abc \Big( \prod_{i=s}^{B-1}b_{\ell_i} \Big)\E
\end{align*}
is a term of the Hall product $ \hprod_{i=1}^n\cO(d_i'+\Tilde{\sigma} d)$. Hence, $\Tilde{\sigma}$ realizes the Hecke modification $[\E' \xrightarrow[]{x} \E] $.

Conversely, let $[\E' \xrightarrow[]{x} \E] $ be a Hecke modification. Since $\E :=  \bigoplus_{i=1}^n \cO(d_i)$ with $d_i \leq d_{i+1}$, by Theorem \ref{menorqued}, we might write $\E' := \bigoplus_{i=1}^n\cO(d_i')$ with $d_i' \leq d_{i+1}'$ such that  $d_i'=d_i-\epsilon_i$ for some $\epsilon_i \in \{1, \ldots, d\}$ for $i=1, \ldots, n$ and $\sum_{i=1}^n \epsilon_i = d$. 

We claim that $\Tilde{\sigma}$ realizes the Hecke modification $[\E' \xrightarrow[]{x} \E]$.
Indeed, let $\sigma_j \in \Delta_{1}^n$ given by 
\[\sigma_j(i)=\left\{\begin{array}{ll}
1, \ \ \ \  \text{if} \  i=j \\ 0,  \ \ \  \ \text{otherwise.}   \end{array}\right.\] Suppose that $d_s-\epsilon_s \neq d_i'=d_i$ for all $i \in \{1, \ldots, s-1\}$.
If there exists $\sigma_j$ for some $j>s$ that realizes the Hecke modification, then in the Hall product 
\[\hprod_{i=1}^n \cO(d_i-\epsilon_i+\sigma_j(i)d)\] 
we have that $d_i-\epsilon_i+\sigma_j(i)d=d_i-\epsilon_i$ are in increasing order for $i\in \{1, \ldots, j\}$. Then, there exists a positive integer $a' \in \Z$ such that 
\[\hprod_{i=1}^n \cO(d_i-\epsilon_i+\sigma_j(i)d) = a' \; \bigoplus_{i=1}^{j-1}\cO(d_i-\epsilon_i)*\hprod_{i=j}^{n}\cO(d_i-\epsilon_i+\sigma_j(i)d).\]
Then, every term in this Hall product must have a line subbundle with degree $d_s-\epsilon_s$, in particular $\E$ must have a line subbundle with degree $d_s-\epsilon_s$. That is, there exists $t \in \{s+1, \dots, n\}$ such that $d_t=d_s-\epsilon_s$. Since the degrees are in increasing order and $\epsilon_s>0$, then $d_t<d_s$ and $t<s$, a contradiction. Therefore,  $j \leq s$.

Suppose that $j<s$. Then, there exists a positive integer $a'' \in \Z$ such that
\[\hprod_{i=1}^n \cO(d_i-\epsilon_i+\sigma_j(i)d) = a'' \; \bigoplus_{i=1}^{j-1}\cO(d_i)*\hprod_{i=j}^{n}\cO(d_i-\epsilon_i+\sigma_j(i)d).\]

If $d_{j+1} > d_j$, then $\cO(d_j+d)*\cO(d_{j+1})$ does not have a term of degree $d_j$. The same happens in the Hall product of line bundles with greater degree, contradicting that $\sigma_j$ realizes the Hecke modification. 

If $d_{j+1}=d_j$, then 
\[ \cO(d_j+d)*\cO(d_{j+1})=q^{d+1} \cO(d_j)\oplus \cO(d_j+d),\]
which has a line subbundle of degree $d_j$. In the Hall product 
$\cO(d_{j+1}+d)*\cO(d_{j+2})$
if $d_{j+2} > d_{j+1}$ we have a contradiction as in the previous case, thus $d_{j+2}=d_{j+1}=d_j$. 

Using a similar argument with \( d_{j+3}, \ldots, d_n \), we eventually obtain \( d_{j+k+1} > d_{j+k} \) (for example, if \( j + k = s \)), leading to a contradiction.
However, since the Hecke modification $[\E' \xrightarrow[]{x} \E]$ there exists, there is an element $\sigma \in \Delta_1^n$ that realizes it. Therefore, the only possible case is $\sigma=\Tilde{\sigma}$.

Next, suppose that $d_s-\epsilon_s=d_{s-1}= d_{s-2} = \cdots = d_{s-k}$ and $d_s-\delta \neq d_i$ for all $i<s-k$.
If either $j>s$ or $j<s-k$, we have the same situation as the previous case. Thus, we are left to the case $j=s-\ell$ for some $\ell \in \{1, \ldots, k\}$. 

In the Hall product $\hprod_{i=1}^n \cO(d_i-\epsilon_i+\sigma_j(i)d)$, there is the following term
\[q^{(d+1)(s-j)} \prod_{i=1}^{s-j-1} \dfrac{q-1}{q^{s-j-i}-1} \hprod_{i=1}^n \cO(d_i-\epsilon_i+\sigma_s(i)d).\]
This is obtained by moving to the right the line bundle \( \mathcal{O}(d_s - \delta + d) \) until all the line bundles of degree \( d_s - \epsilon_s \) are to the left of \( \mathcal{O}(d_s - \delta + d) \) in the Hall product.
Note that this is the only term that allows us recovery $\E$, since any other term does not have $k-1$ line subbundles of degree $d_s-\epsilon_s$ on the left of $\cO(d_s)$.

Hence, given an element $\sigma \in \Delta_1^n$ that realizes $[\E' \xrightarrow[]{x} \E]$, there exists a positive integer $h \in \Z$ such that all terms in the Hall product $\hprod_{i=1}^n \cO(d_i-\epsilon_i+\sigma(i)d)$ that allows us recovery $\E$, are terms of the following Hall product 
\[h \hprod_{i=1}^n \cO(d_i-\epsilon_i+\Tilde{\sigma}_j(i)d). \]
Therefore, $\Tilde{\sigma}$ realizes $[\E' \xrightarrow[]{x} \E ]$, which proves our claim.

The fact that $\Tilde{\sigma}$ realizes the Hecke modification $[\E' \xrightarrow[]{x} \E ]$ implies that $\E$ is a term of the Hall product in equation $\eqref{prodpeso1}$.

We suppose by contradiction that $d_{s+1}-\epsilon_{s+1}>d_s$. Then the Hall product 
$$\cO(d_s-\epsilon_s+d)*\cO(d_{s+1}-\epsilon_{s+1})$$ 
has only terms of degree greater that $\min\{d_s-\epsilon_s+d,d_{s+1}-\epsilon_{s+1}\}>d_s$. As the degrees are in increasing order, all other Hall products would have line subbundles with degree greater than $d_s$, a contradiction. Thus,  $d_{s+1}-\epsilon_{s+1}\leq d_s$ and there exists only a positive integer $b_{\ell_s} \in \Z$ such that
 \[\cO(d_s-\epsilon_s +d)*\cO(d_{s+1}-\epsilon_{s+1})=b_{\ell_s}\cO(d_s)\oplus\cO(d_{s+1}+d-\epsilon_s-\epsilon_{s+1}).\]
 
In the Hall product, $ \cO(d_{s+1}+d-\epsilon_s-\epsilon_{s+1})*\cO(d_{s+2}-\epsilon_{s+2}),$ 
by the same reasoning as before, we obtain $d_{s+2}-\epsilon_{s+2} \leq d_{s+1}$.

Proceeding inductively, using that $d-\sum_{j=s}^{s+i}\epsilon_j>0$ for $i \in \{0, \ldots, B-s-1\}$, we conclude that 
$$d_{s+i+1}-\epsilon_{s+i+1} \leq d_{s+i+1},$$ 
which is the desired conclusion.
\end{proof}


 \section{Application: Hecke Eigenforms}

Let $X$ be a smooth projective curve defined over a finite field $\Fq$. Let $F$ be the function field  of $X$ and $\mathbb{A}$ be its adelic ring. While the definitions hold for every $X$, in the theorems we assume $X$ to be the projective line. Moreover, instead of $\Bun_n (X)$, we consider only projective vector bundles $\P\Bun_n (X)$. For $\E \in \Bun_n (X)$, we abuse the notation and continue to use $\E$ to denote its class on $\P\Bun_n (X)$. The goal of this section is to apply previous calculation to show that the space of unramified toroidal automorphic forms over $\P^1$ for $\PGL_n$ is trivial, for every $n \geq 2$. For $n=2$, this has been previously showed by Lorscheid in \cite[Thm. 10.9]{oliver-graphs}. We refer to \cite{dennis-03}  for the precise definitions and basic facts about unramified automorphic forms.  

In what follows, we apply a theorem due to Weil (cf.\ either \cite[Lemma 5.1.6]{lorscheid-thesis} or \cite[Lemma 3.1]{frenkel-04}) to consider an unramified automorphic form as a complex valuated function
\[ f : \P\Bun_n (X) \longrightarrow \C. \] 
Let $\mathcal{A}$ be the set of all unramified automorphic forms over $X$ for $\PGL_n$. In this setting, the Hecke operator associated to a torsion sheaf $\mathcal{T} \in \mathrm{Tor}(X)$ is the operator  
\[  \Phi_{\mathcal{T}} : \mathcal{A} \rightarrow \mathcal{A} \]
given by  
\[ \Phi_{\mathcal{T}}(f)(\E) := \sum_{\substack{\E' \subseteq \E \\ \E/\E' \cong \mathcal{T}}} f (\E')\]
where the sum runs over the coherent subsheaves $\E'$ of $\E \in \Bun_n (X)$ whose quotient is isomorphic to $\mathcal{T}$. We observe that $\E'$ is locally free sheaf since it is contained in $\E$ and, therefore, $f(\E')$ is well-defined. 
When $\mathcal{T} = \mathcal{K}_{x}^{\oplus r}$, we denote $\Phi_{\mathcal{T}}$ simply by $\Phi_{x,r}$. We also refer to \cite[Sec. 2]{kapranov} for a background material related to automorphic forms in this geometric setting.

\begin{df} Let $\E,\E' \in \Bun_n (X)$ and $\mathcal{T} \in \mathrm{Tor}(X)$. If $\E' \subseteq \E$ and $\E/\E' \cong \mathcal{T}$, we say that $\E'$ is a $\Phi_{\mathcal{T}}$-neighbor of $\E.$ Note that in this case, the class of $\E'$ in $\P\Bun_n (X)$ appears in the summand given by the action of the Hecke operator  $\Phi_{\mathcal{T}}$. 
\end{df}


\begin{df} Let $x \in X$ be a closed point and $\underline{\lambda} := (\lambda_1, \ldots, \lambda_{n-1}) \in \C^{n-1}$. The space of $\Phi_{x,r}$-eigenforms, for $r=1, \ldots, n-1$, with eigenvalues $\underline{\lambda}$ is 
\[ \mathcal{A}(x, \underline{\lambda}) := \big\{ f \in \mathcal{A}\;\big|\; 
\Phi_{x,i}(f) = \lambda_i f \; \text{ for $i=1, \ldots, n-1$} \big\}
\]
where $\mathcal{A}$ is the space of unramified automorphic forms. 
\end{df}


\noindent
\textbf{Notations.} By a \textit{partition of length} $m \in \Z_{>0}$, we mean a sequence $\rho := (d_1, \ldots, d_m)$ of positive integers in non-decreasing order i.e., $d_1 \leq d_2 \leq \cdots \leq d_m$. The $d_i$ are called the \textit{parts} of $\rho$. We denote by $\rho = (1^{\ell_1}, 2^{\ell_2}, \ldots, m^{\ell_m})$ the partition that has exactly $\ell_i$ parts equal to $i$.

By the classification of vector bundles over the projective line, every rank $n$ projective vector bundle can be written as either
\[   \E_0:=  \bigoplus_{i=1}^n \cO  \quad \text{or}  \quad
    \E_{\rho} := \bigoplus_{i=1}^{n-t-\ell}\cO   \oplus \bigoplus_{i=1}^{\ell} \cO(1) \oplus \bigoplus_{i=1}^t \cO(d_i)
\] 
where $\rho = (1^{\ell},d_1, \ldots, d_t)$ is a partition of length $\ell +t$, for $1\leq \ell+t \leq n-1$. 

Given an element $\sigma \in \Delta_s^n$, we keep using the same notation to define 
\[ \sigma:  \Z^n \longrightarrow    \Z^n\] 
given by
\[ 
 (d_1, \ldots, d_n) \longmapsto  (d_1+\sigma(1), \ldots, d_n+\sigma(n)).
\]


\begin{rem} \label{rem-langlands}
Our aim in this section is to apply the results from the previous sections to prove the triviality of the spaces of toroidal and cuspidal forms over \( \mathbb{P}^1 \). To achieve that, we first show that the space \( \mathcal{A}(x, \underline{\lambda}) \) is trivial. 

According to the geometric Langlands correspondence, proved by Drinfeld \cite{drinfeld} for \( n = 2 \) and by Lafforgue \cite{lafforgue-02} for \( n > 2 \), the space $\mathcal{A}$ admits a basis consisting of eigenvectors of the unramified (or spherical) Hecke algebra, labeled by equivalence classes of \( n \)-dimensional representations of the unramified quotient of the Weil group of \( F \). In the case of our interest, i.e., for \( X = \mathbb{P}^1 \), the unramified Weil group is isomorphic to the Weil group of the base field \( \mathbb{F}_q \), which consists of all integer powers of the Frobenius automorphism and is therefore isomorphic to \( \Z \). Roughly speaking, this is because the Weil group of \( F \) for a general curve \( X \) is an extension of the Weil group of \( \mathbb{F}_q \) by the étale fundamental group of \( \overline{X} = X \otimes_{\mathbb{F}_q} \overline{\mathbb{F}}_q \) - and for \( X = \mathbb{P}^1 \), this latter group is trivial since \( \mathbb{P}^1 \) is simply connected, i.e., it has no non-trivial étale covers. Hence, the eigenvectors are labeled by equivalence classes of \( n \)-dimensional representations of \( \Z \), which correspond to collections \( (\lambda_1, \ldots, \lambda_n) \) of nonzero complex numbers (up to permutations). Moreover, these numbers must give the eigenvalues of the Hecke operators \( \Phi_{x,r} \) at all closed points \( x \) of degree one. This follows because, in the unramified Weil group of \( F \), the Frobenius morphism at such \( x \) coincides with the Frobenius morphism of \( \mathbb{F}_q \), the generator of the Weil group of the base field. Under the Langlands correspondence, the eigenvalues of this Frobenius morphism match those of the Hecke operators \( \Phi_x \). Therefore, one concludes  that \( \mathcal{A}(x, \underline{\lambda}) \) is trivial.

While the triviality of \( \mathcal{A}(x, \underline{\lambda}) \) follows from the geometric Langlands correspondence, as explained above, we present a proof of this fact that is significantly more elementary. Thus, even though the following results may be known to experts, the proofs provided below are new and rely solely on the theory of Hecke modifications. Therefore, accessible to those non-experts on Langlands program. 
\end{rem}


 \begin{thm} \label{thm-dim1eigenforms} Let  $x \in \P^1$ be a closed point of degree one, and let $\E \in \PBun_n\P^1$. Let $ \underline{\lambda} =( \lambda_1 , \ldots , \lambda_{n-1}) \in \C^{n-1}$. If $f \in \cA(x, \underline{\lambda})$,  then $f(\E)$ is determined by $ \underline{\lambda} $ and $f (\E_0)$. In
particular,
\[\dim_{\C} \cA(x,\underline{\lambda}) = 1.\]
 \end{thm}

\begin{proof} 
By Theorem $\ref{deg1anyr}$, $\E'$ is in the $\Phi_{x,r}$-neighbor of $\E$ if and only if there exists  $\delta \in \Delta_r^n$ that realizes the modification $[\E' \xrightarrow[r]{x} \E ]$.  Note that, in $\P\Bun_n (\P^1)$,  $\E' \cong \E' \otimes \cO(1)$. Hence,  $\E'$ a $\Phi_{n,r}$-neighbor of $\E$ if and only if $\E'$ correspond to add $1$ in the degree of $n-r$ line subbundles in the decomposition of $\E$. Thus, Theorem \ref{spacedmultip} yields
 \[\lambda_r f(\E_0)=q^{r(n-r)} f(\E_{(1^{n-r})}),\]
 and the theorem is true for any vector bundle isomorphic to a direct sum of line bundles of degree one and zero.

Let $t \in \Z_{> 1}$. Suppose the theorem holds for any projective vector bundle that can be written as a direct sum of $t-1$ line bundles of degree strictly greater than one and $n-t+1$ line bundles of degree zero or one. That is, the statement is true for every projective bundle in the form $\E_{(1^i, d_1, \ldots, d_{t-1})}$ with $i \in \{0, \ldots, n-t+1\}$. We will show that the same is true for every projective rank $n$ vector bundle on the form $\E_{(1^i, d_1, \ldots, d_{t})}$ with $i \in \{0, \ldots, n-t\}$.

Suppose that given positive integers $d_1\leq d_2 \leq  \cdots \leq d_t $,  the theorem is true for every projective bundle that can be written in the form $\E_{(1^{\ell}, d_1', \ldots, d_t')}$ with $d_i' \leq d_i$ for all $i\in \{1, \ldots, t\}$ and $0 \leq \ell <n-t$. 

The base case of induction is the projective bundle $\E_{(1^{\ell},2^t)}$. In this case, looking for the $\Phi_{x,n-1}$-neighbor of $\E_{(1^{\ell},2^{t-1})}$, there exist positive integers $a_1,a_2, a_3 \in \Z$ such that
\[\lambda_{n-1} f(\E_{(1^{\ell},2^{t-1})})=a_1f(\E_{(1^{\ell+1},2^{t-1})})+a_2f(\E_{(1^{\ell},2^{t-2},3)})+a_3f(\E_{(1^{\ell},2^{t})}). \]

We observe that, except by $\E_{(1^{\ell},2^t)}$, every vector bundle in the above equation is of the form $\E_{(1^i, d_1, \ldots, d_{t-1})}$ with $i \in \{0, \ldots, n-t+1\}$, thus the theorem is true in this terms. Hence, $\E_{(1^{\ell},2^t)}$ is determined by the values of $\lambda_1, \ldots, \lambda_{n-1}$ and $f(\E_0)$.

Fix $r \in \{1, \ldots, n-t\}$, the $\Phi_{x,r}$-neighbors of $\E_{(d_1, \ldots, d_t)}$ are given as follows,
\begin{equation}\label{aut}
    \begin{aligned}
        \lambda_r f(\E_{(d_1, \ldots , d_t)})= & \sum_{\sigma \in \Delta_{n-r}^n} a_{r\sigma}f(\E_{(\sigma(d_1, \ldots, d_t))})\\
        =& \sum_{j=0}^{\min\{t,r\}} \sum_{\sigma \in \Delta_{t-j}^t} a_{rj\sigma} f(\E_{(1^{n-t-r+j},\sigma(d_1, \ldots, d_t))})
    \end{aligned}
\end{equation}
for some $a_{r\sigma},a_{rj\sigma} \in \Z_{> 0}$.
The above equation reflects that an element on the $\Phi_x^r$- neighbor of $\E_{(d_1, \ldots, d_t)}$ correspond to increase the degree of $n-r$ line subbundles by one. This is done  by increasing the degree of $n-t-r+j$ line subbundles of degree $0$ and, by increasing the degree of $t-j$ line subbundles of degree greater than zero. Note that two vector bundles in the equation \eqref{aut} can be equal. This happens, for instance, if $d_1=d_2 \neq d_i$, for all $i>2$ and $\sigma_1=(1, 0 \ldots, 0), \ \sigma_2=(0,1,0, \ldots, 0)$. In this case, $$a_{r1\sigma_1}=a_{r1\sigma_2}=\dfrac{m_{x,n-1}(\E_{(d_1, \ldots, d_t)}, \E_{(d_1+1, d_2, \ldots, d_t)})}{2}.$$ 

Let $\hat{\sigma}$ be the unique element of $\Delta_t^t$ and $r=n-t$ in  identity $\eqref{aut}$. Then 
\begin{equation}\label{r=n-t}
    a_{r0\hat{\sigma}}f(\E_{(\hat{\sigma}(d_1, \ldots , d_t))}) = \sum_{j=1}^{\min\{t, n-t\}} \sum_{\sigma \in \Delta_{t-j}^t} a_{rj\sigma} f(\E_{(1^j,\sigma(d_1, \ldots, d_t))}) - \lambda_{n-t} f(\E_{(d_1, \ldots , d_t)}).
\end{equation}

We are left to show that all terms on the right side of equation $\eqref{r=n-t}$ are determined by $\lambda_1, \ldots, \lambda_{n-1}$ and $f(\E_0)$. By hypothesis this is true for $f(\E_{(d_1, \ldots, d_t)})$. 

Let $j_0 \in \{1, \ldots, \min\{t,n-t\}\}$ and $\Tilde{\sigma} \in \Delta_{t-j_0}^t$. The next step is to prove the theorem for $f(\E_{\rho_{j_0 \Tilde{\sigma}}})$, where $\rho_{j_0\Tilde{\sigma}}:= (1^{j_0},\Tilde{\sigma}(d_1, \ldots, d_t))$.
If $j_0=n-t$, then 
\[ \E_{(1^{j_0},\Tilde{\sigma}(d_1, \ldots, d_t))}=\E_{(1^{j_0},\Tilde{\sigma}(d_1, \ldots, d_t))}\otimes \cO(-1)=\E_{(\Tilde{\sigma}(d_1-1, \ldots, d_t-1))}.\] 
Since $d_i-1+\Tilde{\sigma}(i) \leq d_i$, for all $i \in \{1, \ldots, t\}$, the theorem holds in this terms by hypothesis.
If $j_0\leq n-t$,  let $r = n-t-j_0$ in the identity \eqref{aut}. Thus

\begin{equation}\label{n-t-j_0}
        a_{r0\hat{\sigma}}f(\E_{(1^{j_0}\hat{\sigma}(d_1, \ldots , d_t))}) = \sum_{j=1}^{\min\{t, n-t-j_0\}} \sum_{\sigma \in \Delta_{t-j}^t} a_{rj\sigma} f(\E_{\rho_{rj\sigma}}) - \lambda_{n-t-j_0} f(\E_{(d_1, \ldots , d_t)}),
\end{equation}
where $\rho_{rj\sigma}=(1^{j+j_0},\sigma(d_1, \ldots d_t))$. Note that 
\[a_{r0 \Tilde{\sigma}} \; f(\E_{\rho_{j_0 \Tilde{\sigma}}}) =a_{r0 \Tilde{\sigma}} \; f(\E_{(1^{j_0},\hat{\sigma}(\Tilde{\sigma}(d_1-1, \ldots , d_t-1)))}).\]
Then, replacing $\E_{(d_1, \ldots, d_t)}$ by $\E_{\Tilde{\sigma}(d_1-1, \ldots, d_t-1)}$ in equation \eqref{n-t-j_0} we obtain:         
\begin{equation}\label{r=n-t-j}
       a_{r0\Tilde{\sigma}} \; f(\E_{\rho_{j_0 \Tilde{\sigma}}})=\sum_{j=1}^{\min\{t,n-t-j_0\}} \sum_{\sigma \in \Delta_{t-j}^{t}} b_{rj\sigma \Tilde{\sigma}} f(\E_{\rho_{rj\sigma\Tilde{\sigma}}})  - \lambda_{n-t} f(\E_{(\Tilde{\sigma}(d_1-1, \ldots , d_t-1))}),
\end{equation}
where $\rho_{rj\sigma\Tilde{\sigma}} := ((1^{j+j_0},\sigma(\Tilde{\sigma}(d_1-1, \ldots, d_t-1))))$ and $b_{ri\sigma \Tilde{\sigma}} \in \Z \setminus \{0\}$.

Note that if $\Tilde{\sigma}(\sigma(d_\ell-1))=1$, for some index $\ell \in \{1, \ldots , t\}$ and some $\sigma \in \Delta_{t-\ell}^t$, then  $\E_{(1^{j_0+\ell},\sigma(\Tilde{\sigma}(d_1-1, \ldots, d_t-1)))}$ can be represented as $\E_{(1^{j_0+\ell+1},d_1', \ldots, d_{t-1}')}$, and by hypothesis the theorem holds for this vector bundle.

Furthermore, if $j=n-t-j_0$, then \[\E_{(1^{j+j_0},\sigma(\Tilde{\sigma}(d_1-1, \ldots, d_t-1)))} \cong \E_{(1^{n-t},\sigma(\Tilde{\sigma}(d_1-1, \ldots, d_t-1)))}\otimes \cO(-1)  \cong \E_{(\sigma(\Tilde{\sigma}(d_1-2, \ldots, d_t-2)))}.\]
Hence, for all $\ell \in \{1, \ldots, t\}$, we obtain $\sigma(\Tilde{\sigma}(d_{\ell}-2))\leq d_{\ell}$. Thus the theorem holds in this case.

Next, we might proceed recursively in each term of the right side of equation $\eqref{r=n-t-j}$. The recursion eventually ends,  since if $\sigma_1(\sigma_2( \ldots \sigma_s(d_{\ell}-s)))$ is not smaller than $d_{\ell}$ for all $\ell \in \{1, \ldots, t\}$, the power $i+j_1+ \ldots +j_s$ eventually become $n-t$. Then 
\[ \sigma_1(\sigma_2( \ldots \sigma_s(d_{\ell}-s-1)))\]
is smaller than $d_{\ell}$, for all $\ell \in \{1, \ldots, t\}$. 

Given $j \in \{1, \ldots, t-1\}$ and $\sigma \in \Delta_{t-j}^t$, by a changing of variables in the above process, i.e., 
\[ (d_1, \ldots, d_t) \mapsto \sigma(d_1-1, \ldots, d_t-1)\]
the theorem holds for every vector bundle $\E_{\overline{\sigma} (d_1, \ldots, d_t)}$, where 
$\overline{\sigma}(i)=1-\sigma(i)$. 
Hence, the theorem holds for every vector bundle on the form $\E_{\hat{d_1}, \ldots, \hat{d_t}}$, with $\hat{d_1}, \ldots, \hat{d_t} \in \Z_{>1}$.

Finally, to prove that the theorem holds for a vector bundle of the form $\E_{(1)^{\ell},\hat{d_1}, \ldots \hat{d_t}}$ with $\ell \in \{1, \ldots, n-r-1\}$,  we just need to consider in equation $\eqref{r=n-t}$, the case $r=n-t-\ell$ and $j=0$ in the variables $(\hat{d_1}-1, \ldots, \hat{d_t}-1)$. Since we proved that the theorem holds for every term of the equation \eqref{r=n-t} and for all $r \in \{1, \ldots, n-t\}$, it is also holds for $\E_{(1^{\ell},\hat{d_1}, \ldots \hat{d_t})}$. This completes the proof.
\end{proof}


\begin{df}
    An unramified automorphic form $f \in \cA$ is a \textit{cusp form} if for any  $r , s \in \Z_{ > 0}$ with $r + s = n$ and any vector bundles $\F \in \Bun_{r} \P^1$, $\G \in \Bun_{s}\P^1$, 
\begin{equation}\label{cuspeq}
 \sum_{\E \in \Ext(\F,\G)} f(\E) =0,
\end{equation}
where we abuse the notation and write $\E$ meant the middle term of the correspondent exact sequence. We denote the space of unramified cusp forms by $\mathcal{A}_{0}$.
\end{df}


\begin{cor} For every $n \in \N$, there are no unramified cusp forms for $\PGL_n$ over the projective line. 
\end{cor}

\begin{proof} Let $T$ be the diagonal torus of $\GL_n$, $\underline{\lambda} =( \lambda_1 , \ldots , \lambda_{n-1}) \in \C^{n-1}$ and $x \in \P^1$ be a closed point of degree one. 
By \cite[Thm. 1.7.7]{valdir-thesis}, there exists a nontrivial Eisenstein series induced from a unramified character of $T$, which is an eigenfunction for $\Phi_{x,r}$ ($r=1, \ldots, n-1$) with eigenvalues $\lambda_1 , \ldots , \lambda_{n-1}$. 
Hence,  \cite[Thm. 1.7.9]{valdir-thesis} and above theorem yield 
\begin{equation} \label{eq-tr}
\mathcal{A}_{0} \cap   \mathcal{A}(x, \underline{\lambda}) = \{0\}.
\end{equation}
Since $\mathcal{A}_{0}$ splits as a direct sum of irreducible representations,  we can write every $f \in \mathcal{A}_{0}$ as a sum of eigenforms. Therefore, $f=0$ by above (\ref{eq-tr}). \end{proof}


In the following, we use the adelic interpretation of an unramified automorphic form i.e., as a complex valued function 
\[ f: \GL_n(F) Z(\mathbb{A}) \setminus \GL_n(\mathbb{A}) / \GL_n(\cO_{\A})  \longrightarrow \C\]
with some moderate growth condition. We observe that the aforementioned theorem due to Weil establishes, for every $n \geq 1$, a bijection 
\[ \GL_n(F) Z(\mathbb{A}) \setminus \GL_n(\mathbb{A}) / \GL_n(\cO_{\A})  \longleftrightarrow  \P\Bun_n(X) \]
where $Z(\mathbb{A})$ is the center of $ GL_n(\mathbb{A}).$

Let $E/F$ be a separable field extension of degree $n$ and $\mathbb{A}_{E}$ be its adelic ring. 
Choosing a basis for $E$ over $F$ gives an embedding of $E^{*}$ in $\GL_n(F)$ and a non-split maximal torus $T \subseteq \GL_n$ with $T(F)= E^{*}$ and $T(\mathbb{A}) = \mathbb{A}_{E}^{*}$. In this case, we say that $T$ is associated to $E/F$. We refer \cite[Def. 1.5.1]{lorscheid-thesis} to the definitions of non-split and maximal torus. 

\begin{df} Let $T$ be a maximal torus of $\GL_n$ over $F$ associated with a separable extension $E/F$ of degree $n$. Endow $T(\mathbb{A})$ and $T(F)Z(\mathbb{A}) $ with the Haar measures  and $T(F)Z(\mathbb{A})\setminus T(\mathbb{A})$ with the quotient measure. For $f \in \mathcal{A}$, we define 
\[f_T(g) := \int_{T(F)Z(\mathbb{A})\setminus T(\mathbb{A})} f(tg)dt\]
the \textit{toroidal integral} of $f$ along $T$.
\end{df}

\begin{rem} The quotient $T(F)Z(\mathbb{A})\setminus T(\mathbb{A}) $ is compact, see \cite[Pag. 42]{valdir-thesis}. \end{rem}

\begin{df} \label{def-toroidal} Let $T$ be a maximal torus of $\GL_n$ over $F$ associated with a separable extension $E/F$ of degree $n$. We define
\[\mathcal{A}_{tor}(E) := \big\{f \in \mathcal{A} \,\big|\, f_T(g)=0, \ \text{for all} \ g \in \GL_n(\mathbb{A})\big\}\]
the space of $E$\textit{-toroidal automorphic forms}, and
\[\mathcal{A}_{tor} = \bigcap_{E/F} \mathcal{A}_{tor}(E)\]
the space of \textit{toroidal automorphic forms}, where $E/F$ runs over the separable extensions of degree $n$. 
\end{df}

\begin{rem} The spaces $\mathcal{A}_{tor}(E)$ do not depend on the choice of the basis for $E$ over $F$, see \cite[Rem. 2.1.3]{valdir-thesis}. \end{rem}

\begin{thm}\label{thm-5.5} Let  $x \in \P^1$ be a closed point of degree one  and  $E$ be the constant field extension of $F$ of degree $n$, i.e.\ $E=\mathbb{F}_{q^n}F$, for $n \geq 2$.  Let $ \underline{\lambda} =( \lambda_1 , \ldots , \lambda_{n-1}) \in \C^{n-1}$. Then
\[\mathcal{A}_{tor}(E) \cap \mathcal{A}(x,\underline{\lambda}) = \{0\}\]
i.e., there do not exist any nontrivial toroidal $\Phi_{x,r}$-eigenforms for $r=1, \ldots, n-1.$ 
\end{thm}

\begin{proof}Let $T$ be the $n$-dimensional torus associated to $E/F$, where $E = \mathbb{F}_{q^n}F.$ Thus $E$ is the function field of $\mathbb{P}_{n}^{1} := \mathbb{P}^1 \otimes_{\mathrm{Spec}\Fq} \mathrm{Spec} \mathbb{F}_{q^n}$, the $n$-th extension of scalars of $\mathbb{P}^1$. The extension of scalars yields 
\[p: \mathbb{P}_{n}^{1} \rightarrow \P^1\]
the projection map. 
Moreover, $p$ induces the inverse image $p^{*} : \Bun_n \P^1 \rightarrow \Bun_n \mathbb{P}_{n}^{1}$ and the direct image (or trace) $p_{*}: \Bun_1 \mathbb{P}_{n}^{1} \rightarrow \Bun_n \mathbb{P}^{1}$. 
From \cite[Thm. 2.8.1]{valdir-thesis},
\[ \int_{T(F)Z(\mathbb{A})\setminus T(\mathbb{A})} f(tg)\mu(t) = c_T \cdot \sum_{[\Line] \in \mathrm{Pic} \mathbb{P}_{n}^{1}/p^{*}\mathrm{Pic}\mathbb{P}^{1}} f(p_{*}\Line)\]
where  
$$c_T = \frac{vol(T(F)Z(\A)\setminus T(\A))}{\# (\Pic(\mathbb{P}_{n}^{1})/p^*\Pic (\P^1))}.$$ 
Hence, if $f \in \mathcal{A}(x, \underline{\lambda})$ is $E$-toroidal, thus  
$f(\E_0) =0$. 
Therefore, Theorem \ref{thm-dim1eigenforms} yields $f(\E_0)=0$ i.e.\ $f$ is trivial.
\end{proof}

We end this article with a proof that the space of toroidal automorphic forms is trivial over $\P^1$ for $\PGL_n$, for every $n \geq 2$.  

\begin{thm} \label{thm-toroidal} Let $F$ be the function field of $\P^1$ over $\Fq$. Let $E$ be the constant field extension of $F$ of degree $n$, i.e.\ $E=\mathbb{F}_{q^n}F$, for $n \geq 2$. Then,
$\mathcal{A}_{tor}(E)  = \{0\} $ and, therefore, 
$\mathcal{A}_{tor} = \{0\}$, for every $n \geq 2$.
\end{thm}

\begin{proof} Assume by contradiction that $\mathcal{A}_{tor}(E) \cap \mathcal{A} \neq \{0\}.$ Hence, let $f \in \mathcal{A}_{tor}(E) \cap \mathcal{A}^K$ be such that $f \neq 0.$ Let 
\[ \mathcal{H} \cdot f := \{ \Phi_{\mathcal{T}} (f) \;|\; \mathcal{T} \in \mathrm{Tor}(\P^1)\}. \]
By admissibility condition, $\mathcal{H} \cdot f$ is a finite dimensional vector space and invariant by the action of $\Phi_{x,r}$ for $r = 1, \ldots, n-1$ and for some $x \in \mathbb{P}^1$ closed point of degree one. Thus, there exists a $\Phi_{x,r}$-eigenform  in $\mathcal{H} \cdot f$, which disagrees with Theorem \ref{thm-5.5}. 
\end{proof}


\section*{Acknowledgments} 
This work was financed, in part, by the São Paulo Research Foundation (FAPESP), Brazil. Process Number 2022/09476-7. The authors would like to express their sincere gratitude to Edward Frenkel, for his careful and detailed explanation of how the triviality of the space of Hecke eigenforms follows from the geometric Langlands correspondence.

\bibliographystyle{alpha}
\bibliography{Hall_alg_and_Hecke_mod} 

\end{document}